\documentclass{article}
\usepackage{hyperref} 
\usepackage[a4paper,margin=1in]{geometry}
\usepackage{amsmath,amssymb,amsthm,mathtools,hyperref}
\usepackage{enumitem}
\usepackage{xcolor}
\usepackage{verbatim}
\hypersetup{hidelinks}

\usepackage[backend=biber,style=numeric,sorting=nyt]{biblatex}
\newtheorem{theorem}{Theorem}
\newtheorem{example}{Example}
\newtheorem{definition}{Definition}
\newtheorem{proposition}{Proposition}
\newtheorem{lemma}{Lemma}

\newtheorem{remark}{Remark}

\newtheorem{corollary}{Corollary}

\newtheorem{letterthm}{Theorem}

\newcommand{\Sone}{\mathbb S^1}
\newcommand{\C}{\mathbb C}
\newcommand{\R}{\mathbb R}
\newcommand{\N}{\mathbb N}
\newcommand{\A}{\mathcal A}
\newcommand{\Pcal}{\mathcal P}
\newcommand{\Leb}{\mathrm{Leb}}
\newcommand{\DimH}{\dim_{\mathrm H}}
\newcommand{\Z}{\mathbb Z}

\newcommand{\Pq}{P^q}

\newcommand{\Jac}{\operatorname{Jac}}

\newcommand{\ba}{\mathbf a}
\newcommand{\bb}{\mathbf b}
\newcommand{\bc}{\mathbf c}
\newcommand{\be}{\mathbf e}
\newcommand{\bt}{\mathbf t}
\newcommand{\bs}{\mathbf s}
\newcommand{\bD}{\mathbf D}

\newcommand{\cB}{\mathcal{B}}
\newcommand{\cD}{\mathcal{D}}
\newcommand{\cF}{\mathcal{F}}
\newcommand{\cW}{\mathcal{W}}

\newcommand{\norm}[1]{\left\Vert#1\right\Vert}

\DeclareMathOperator\supp{supp}

\title{Geometric Properties of Higher Dimensional Solenoidal Attractors}
\author{
		\scshape
		Daniele Galli\thanks{Laboratoire de Mathématiques AGM, Cergy Paris Université, 2 Avenue Adolphe
Chauvin, 95302 Cergy-Pontoise Cedex, France. E-mail: \href{mailto:daniele.galli@cyu.fr}{\texttt{daniele.galli@cyu.fr}}}\\
        \scshape Eleonora Passaglia\thanks{Universität Zürich, Institut für Mathematik, Winterthurerstrasse 190,
			CH-8057 Zürich, Switzerland.\\
			E-mail: \href{mailto:eleonora.passaglia@math.uzh.ch}{\texttt{eleonora.passaglia@math.uzh.ch}}}\\
           \scshape Andrea Ulliana\thanks{École Polytechnique, CMLS, Route de Saclay, 91128 Palaiseau Cedex, France.\\ \textit{E-mail:}\href{mailto:andrea.ulliana@polytechnique.edu}{\texttt{andrea.ulliana@polytechnique.edu}
        }}\\}
\date{October 2026}

\begin{document}

\maketitle

\begin{abstract}
\noindent We study the skew product systems $T:\Sone\times\R^d\to\Sone\times\R^d$,
\begin{equation*}
    T(x,y)=(\ell x, A y+\phi(x)),
\end{equation*}
where $\ell\geq 2,$ $A \in GL_d(\R)$, with $\rho(A)<1$, and $\phi \in C^r(\Sone, \R^d)$, with $r>(d+3)/2$. We allow the fibers to have any dimension and $A$ to be non-conformal. We show that: when $|\det(A)|\ell<1$, for almost every $\phi$, the Hausdorff dimension of the solenoid attractor and the SRB measure equal 1+ the affinity dimension; when $|\det(A)|\ell>1$, for almost every $\phi$, the SRB measure is absolutely continuous. In the latter case we investigate the Sobolev regularity of the density of the SRB measure. We introduce a derivative dispersion condition, replacing the usual transversality condition.

\end{abstract}

\section{Introduction}\label{section: intro}

Given a smooth manifold $M$ and a dynamical system $f:M\rightarrow M$, its physical measures are those which raise naturally from a statistical study of the system. More precisely, an $f$-invariant measure $\mu$ is physical if
\[
\text{vol}\bigg(\Big\{x\in M : \frac 1N \sum _{i=0}^{N-1}\delta_{f^i(x)} \rightarrow \mu \text{ weakly} \Big\} \bigg)>0.
\]
Classifying such measures for given systems is a central problem in smooth dynamics \cite{palis2000global, tsuj2005physical}.
The SRB measures of a hyperbolic system are the invariant measures which are absolutely continuous along the unstable foliation \cite{young2002srb}, i.e. the most compatible with the volume of $M$, when that is not preserved. For Anosov maps the physical measures coincide with the SRB ones \cite{ruelle1976measure,bowen1975equilibrium}. For uniformly hyperbolic systems, the existence of such measures was proved by Sinai \cite{sinai1972gibbs}, and, in a modern perspective, their statistical properties are well treated by the theory of Transfer Operators and Anistropic Banach Spaces (see section \ref{section: anisotropic spaces}), developed, for example, in \cite{blank-keller-liverani-02,liverani-gouezel-06, gouezel-liverani-08,baladi-tsujii-06, baladi2007anisotropic}. Roughly speaking, the elements of anisotropic spaces behave as smooth functions along the unstable directions and as distributions along the stable ones. We underline that, even if the SRB measure is usually singular along the stable directions, there are examples for which it is smooth, or at least absolutely continuous with respect to Lebesgue. It is the case, for instance, of hyperbolic automorphisms of tori, which preserve Lebesgue measure. \newline

The aim of this paper is to investigate the regularity along the stable foliation of the SRB measure in a uniform skew-product setting. Over the past few years, a multitude of anisotropic Banach spaces have been developed, each specifically tailored to the underlying dynamical system \cite{baladi2017quest}. This paper is not an exception and introduces a new anisotropic Sobolev space in the spirit of \cite{baladi2007anisotropic,faure2011upper}. The objects of such spaces behave like Sobolev functions of variable exponent (possibly negative) along different directions. These spaces, when adapted to our setting and after an appropriate tuning of the exponents, represent a unified set of tools to detect any kind of stable regularity: the Hausdorff dimension, absolute continuity, and the smoothness of the density.\newline

We focus specifically on the SRB measure corresponding to the solenoidal attractor of the following family of skew-product systems, with no restrictions on the fiber dimension. For an integer $\ell\geq2$, a matrix $A\in GL_d(\R)$ with spectral radius $\rho(A)<1$ and a function $\phi\in  C^r(\Sone,\R^d)$ for $r> (d+3)/2$, we consider the map 
\[
    T:\Sone\times\R^d\to\Sone\times\R^d,
    \qquad
    T(x,y)=(\ell x, A y+\phi(x)).
\]
$T$ is an Anosov endomorphism, in fact it is a skew product over the expanding map $\tau : x \mapsto \ell x$ and it is definitely uniformly contracting along the fibers.  
This system admits a unique solenoidal attractor $\Lambda^{\phi}$ and a unique SRB measure $\mu_\phi$. We denote by $\Lambda^\phi_x$ and $\mu_{\phi,x}$ their intersection with and conditional measure on the vertical fiber $\{x\}\times\R^d$, respectively.\newline

In this setting, a very explicit description of $\Lambda^{\phi}$ and  $\mu_\phi$ is available. For the alphabet $\A=\{1,\dots, \ell\}$, we associate to any infinite word $\ba\in\A^{\infty}$ a differentiable function $S(\cdot,\ba):\Sone\rightarrow \R^d$, possibly discontinuous at $0$. Then
\[
	\Lambda^{\phi} = \bigcup_{\ba\in\A^{\infty}} \text{graph}\big(S(\cdot,\ba)\big).
\]
Moreover, defining $\Xi(x,\ba):=(x,S(x,\ba))$ and denoting by $m$ the Lebesgue measure on $\Sone$ and by $\nu$ the Bernoulli measure on $\A^\infty$, we have that $\supp(\mu_\phi)=\Lambda^\phi$ and that
\[
\mu_\phi=\Xi_*(m\times\nu).
\]

The scaling of the system alone imposes general bounds on the regularity of $\mu_\phi$ and $\Lambda^\phi$. For instance the maximal possible Hausdorff dimension of both objects is prescribed by Bowen Formula \cite{bowen1979hausdorff,ruelle1982repellers}, that, in our case, can be formulated in terms of the affinity dimension, introduced by Falconer \cite{falconer1988hausdorff}. Let $1>\rho_1\geq\cdots\geq\rho_d>0$ be the moduli of the eigenvalues of $A$, with algebraic multiplicity. For $s\in[k,k+1]$, we set
\[
  \Phi(s):=\rho_1\cdots\rho_k\rho_{k+1}^{\,s-k}.
\]
The affinity dimension $\dim_L(A,\ell)$ is the solution of the equation
$\ell\Phi(s)=1$  in $s$, capped at $d$.

\begin{proposition}\label{upper bound on dim h}
 	For $A \in GL_d(\R)$ with $\rho(A)<1$ and $\ell \geq 2$, for every $x\in\Sone$,
    \[
    \DimH(\Lambda^\phi_x)\leq \dim_L(A, \ell)
    \qquad \text{ and } \qquad
    \DimH(\Lambda^{\phi})\leq 1+\dim_L(A, \ell).
    \]
\end{proposition}

The regime $\dim_L(A, \ell)<d$, where these estimates are interesting, corresponds to having $\ell|\det A|<1$. This takes the name of \emph{thin} regime, where the area contraction naturally prevents any form of absolute continuity of $\mu_\phi$. The opposite case, where $\ell|\det A|>1$, goes by the name of \emph{fat} regime. 
\newline

In the literature, it has been common to assume $A$ to be conformal or, otherwise, to make the following assumption (hyper-thin regime): $A$ can be decomposed as $A=A_{\text{conf}}\oplus A_{s}$ where $A_{\text{conf}}$ is conformal, it contracts less than $A_{s}$ and it holds that $\ell|\det(A_{\text{conf}})|< 1$. In these two situations, the affinity dimension reduces to the more familiar expression of the self-similar dimension, that is, $\dim_L(A,\ell)=-\log(\ell)/\log(|\det(A_{\text{conf}})|)$. The exact computation of the Hausdorff dimension of systems with some degree of affine-similarity is an active branch of dimension theory, and historically has been found to be harder than for (conformally) self-similar systems (e.g., in the case of IFS, an equivalent of Moran–Hutchinson theorem \cite{hutchinson1981fractals}, does not hold in the general affine case, as a suitable Bedford–McMullen \cite{mcmullen1984hausdorff} carpet provides a counterexample).\newline

It is easy to see that the dimension (or regularity) of $\mu_{\phi}$ drops if there is a significant overlap between the graphs of the functions $S(x,\ba)$. They might even all coincide if $T$ admits an invariant curve, which then must coincide with the attractor $\Lambda^{\phi}$. A remarkable phenomenon and a recurring theme in the literature, known since the work of Simon \cite{simon1993hausdorff}, is that, in the absence of overlap, the dimension should commonly not drop from the scaling bound.\newline

Most of the related works capture the geometric independence of the functions $S(x,\ba)$ with similar variations of a \emph{transversality condition}. Essentially, it states that all (or a large proportion of) the graphs of $S(x,\ba)$ are transversal. It is important to note that this condition can only hold if the dimension of the fiber does not exceed the one of the base. In the hyper-thin regime, after a suitable projection of the system, many authors reduce this dimensional constraint to requiring that $\dim(A_{\text{conf}})$ must not exceed the one of the base dynamics.\newline

The study of the thin or the fat case splits into two separate lines of research.
Dealing with the thin regime, most authors focus on the hyper-thin regime only, requiring $\dim(A_{\text{conf}})$ not to exceed the dimension of the base dynamics, and work with the \emph{intrinsic transversality condition}. Bothe \cite{bothe1995hausdorff} formulated the \emph{intrinsic transversality condition}, showing that it ensures no drop of the expected Hausdorff dimension for one dimensional base dynamics. Simon \cite{simon1997hausdorff} verified it for the Smale-Williams solenoid, and Bortolotti and da Silva \cite{bortolotti-silva-22} extended the results of Simon to higher dimensional base dynamics. These authors are able to let $A$ depend on $(x,y)$, although, managing to prove the genericity of the intrinsic transversality condition only under certain uniform constraints on $\|A(x,y)\|$.\newline
The full thin regime has been considered in the remarkable pair of papers by Ren \cite{ren-24,ren2024dichotomy}, for analytic $\phi$ and conformal $A$. It has to be noted that in \cite{ren2024dichotomy} the fiber is two-dimensional and the base is one-dimensional. To treat the higher dimensional fiber case, they formulate an alternative to the transversality condition, the $(\mathbf{H})$-condition, strongly adapted to the analytic setting. This assumption, joint with the analiticity, guaranties no drop in the Hausdorff dimension. We believe this condition not to be sufficient in the smooth setting.\newline
For the fat regime, the seminal work is due to Tsujii \cite{tsujii-01}, who introduced a quantitative transversality condition which is shown to imply absolute continuity of the SRB measure and to be prevalent in a measure theoretic sense. In \cite{avila-gouezel-tsujii-06} sharp Sobolev estimates are proved, and in \cite{bocker-bortolotti-castro-24} the higher dimensional base case is treated. All of these works require the fibers to have at most the same dimension as the base.\newline

Our work, stemming from the fat case line, especially from \cite{tsujii-01} and \cite{avila-gouezel-tsujii-06}, aims at unified treatment of the thin and fat regime, removing any dimensional or conformality constraint, for $C^k$ systems.\newline
The main obstacle is represented by the fact that, both the transversality condition and the $(\mathbf{H})$-condition do not fit our setting. We introduce the \emph{derivative dispersion condition} (or \emph{$(\bD)$-condition} for short), defined precisely in section \ref{section: setting}. It is a quantitative condition which roughly requires that, for any $x$, only a small amount of the derivatives $\frac d{dx}S(x,\ba)$ can lie near any hyperplane of $\mathbb{R}^d$.\newline

\begin{letterthm}
\label{thm:A}
    Given $A \in GL_d(\R)$ with $\rho(A)<1$ and $\ell \geq 2$, if $|\det(A)|<1/\ell$ and if $\phi\in C^{r}(\Sone,\mathbb{R}^d)$ with $r > (d+3)/2$ satisfies the derivative dispersion condition, then 
    $\DimH(\mu_{\phi})=1+\dim_L(A, \ell).$
\end{letterthm}

\begin{letterthm}
\label{thm:B}
    Given $A \in GL_d(\R)$ with $\rho(A)<1$ and $\ell \geq 2$ if $|\det(A)|> 1/ \ell$ and if $\phi\in C^{r}(\Sone,\mathbb{R}^d)$ with $r>(d+3)/2$ satisfies the derivative dispersion condition, then $\mu_{\phi}$ is absolutely continuous with respect to the Lebesgue measure on $\Sone \times \R^d.$ Moreover, its density $\Psi_{\phi}$ belongs to the Sobolev space $H^{s}(\Sone\times\mathbb{R}^d)$ for any $s>0$ satisfying $|\det(A)|\rho(A^{-1})^{-2s}>1/\ell$ and $r>s+(d+3)/2$.
\end{letterthm}

The derivative dispersion condition is generic, in fact it holds for almost every $\phi \in C^r(\Sone, \R^d)$ in the following sense. We say that $F \subset C^r(\Sone, \R^d)$ has full measure if there exist finitely many smooth functions $\psi_1,...,\psi_m \in C^{\infty}(\Sone, \R^d)$ such that, for every $\phi \in C^r(\Sone, \R^d)$, 
\begin{equation*}
    \text{Leb} \left \{ (t_1,...,t_m) \in \R^m : \phi + \sum_{i=1}^m t_i \psi_i \notin F \right\} =0.
\end{equation*}

\begin{letterthm}\label{thm:C}
    Given $A \in GL_d(\R)$ with $\rho(A)<1$ and $\ell \geq 2$, we denote by $\mathcal{F} \subset C^r(\Sone, \R^d)$ the set of $\phi$ that satisfies the derivative dispersion condition. Then the interior of $\mathcal{F}$ has full measure in $C^r(\Sone, \R^d).$
\end{letterthm}

In section \ref{section: setting} we discuss more general versions of Theorems \ref{thm:A} and \ref{thm:B} involving a relaxation of the $(\bD)$-condition. Indeed, for any $\beta>0$, we formulate the $(\bD_\beta)$-condition, admitting an higher number of exceptions than in the $(\bD)$-condition. This allows to prove non-sharp, but relevant results, such as the following one.

\begin{example}\label{example times 3}
    The Smale-Williams type map $T:\Sone\times \R^2\to\Sone\times\R^2$ given by $T(x,y,z) = \big(3x, \lambda y + \cos(2\pi x), \lambda z + \sin(2\pi x)\big)$, has absolutely continuous SRB measure, and positive Lebesgue measure attractor for all $\sqrt{2/3}<\lambda<1$.
\end{example}

\paragraph{Structure of the Paper} In section \ref{section: setting} we fully describe our setting and the tools required for our discussion, among them we formulate the derivative dispersion condition and we define our adapted anisotropic Sobolev norms. Finally, we state Theorems \ref{thm:1}, \ref{genericity} and \ref{thm:3} which describe our results in full generality. In section \ref{section: proof of intro} we prove the statements of the introduction, assuming the ones of section \ref{section: setting}. Section \ref{section: energy and sobolev} recalls elements of potential theory, describes the relation between Sobolev Spaces and Hausdorff Dimension and proves Theorem \ref{thm:3}. In section \ref{section: anisotropic spaces} we recall preliminary notions about transfer operators and classical anisotropic spaces. In section \ref{section: proof of theorem 1} we prove a Lasota-Yorke type inequality involving the anisotropic Sobolev norm and the classical ones, obtaining Theorem \ref{thm:1} as a corollary. Section \ref{section: proof of genericity} is devoted to the proof of Theorem \ref{genericity}, i.e. the genericity of the derivative dispersion condition.

\paragraph{Declaration of AI Usage}

During this work ChatGPT 5.3-5.6 and Gemini 3.6 have been used for literature research, computation and polishing of the manuscript. The authors take full responsibility of the content of this paper.

\paragraph{Acknowledgments}

We warmly thank Artur Avila and Zhiyuan Zhang for the several useful conversations and suggestions. This project is part of the authors' activity within the UMI Group DinAmicI.

\section{Setting and Results}\label{section: setting}

We recall that, throughout the paper, we consider the skew product
\[
    T(x,y)
    =
    \bigl(\tau(x),Ay+\phi(x)\bigr),
    \qquad
    \tau(x)=\ell x \pmod 1,
\]
where \(\ell\geq2\), \(A\in GL_d(\mathbb R)\) satisfies
\(\rho(A)<1\), and
\(\phi\in C^r(\mathbb S^1,\mathbb R^d)\), where $r>(d+3)/2$.

\subsection{Symbolic Setting}
\paragraph{Symbolic notation.}

Let $\A:=\{1,\dots,\ell\},$ and
for $n\geq 1$, let $\A^n$ be the set of words of length $n$ over the alphabet
$\A$. If $\ba\in\A^n$, we write $\ba=(a_1,\dots,a_n).$
The space of infinite words is denoted by $\A^\infty:=\A^{\mathbb N}.$\\
Let $\Pcal$ be the partition of $\Sone$ into the intervals $\Pcal(k):=[(k-1)/\ell,k/\ell)$ with $1\leq k\leq \ell.$ For a word $\ba=(a_1,\dots,a_n)\in\A^n$ we have
\[
    \Pcal(\ba)=\bigcap_{j=0}^{n-1}\tau^{-j}(\Pcal(a_{n-j})).
\]
For every $x\in\Sone$ and every $\ba\in\A^n$, there is a unique point
$\ba(x)\in \Pcal(\ba)$ such that $\tau^n(\ba(x))=x.$
For $1\leq q\leq n$, we denote by $[\ba]_q=(a_1,\dots,a_q)$
the prefix of length $q$ of $\ba$.

\paragraph{Graph functions}

Fix a word $\ba=(a_1,\dots,a_n)\in\A^n$. For $1\leq j\leq n$, we define
\[
    \tau_{\ba}^{-j}:\Sone\to\Pcal([\ba]_j)
\]
as the inverse branch of $\tau^j$ selected by the prefix $[\ba]_j$, so $\tau_{\ba}^{-j}(x)$ is the unique point in
$\Pcal([\ba]_j)$ such that $\tau^j\bigl(\tau_{\ba}^{-j}(x)\bigr)=x$. For $\ba\in\A^n$, let
\[
    S_n(x,\ba)
    :=
    \sum_{j=1}^n
    A^{j-1}
    \phi\bigl(\tau_{\ba}^{-j}(x)\bigr).
\]

The image of the horizontal segment $\Pcal(\ba)\times\{0\}$ under $T^n$ is the
graph of $S_n(\cdot,\ba)$, in particular $T^n\bigl(\tau_{\ba}^{-n}(x),0\bigr)
    =
    \bigl(x,S_n(x,\ba)\bigr)$.
For an infinite word $\ba\in\A^\infty$, we introduce
\[
    S(x,\ba)
    :=
    \lim_{n\to\infty}S_n(x,[\ba]_n)  = 
    \sum_{j=1}^\infty
    A^{j-1}
    \phi\bigl(\tau_{\ba}^{-j}(x)\bigr).
\]
Since \(\rho(A)<1\), the series converges uniformly, together with its
derivatives up to order \(r\).

\paragraph{Bounds on Derivatives}
 
Let $\bc\in\A^p$. We denote by $\Pcal_*(\bc)$ the union of the interval $\Pcal(\bc)$
and the two adjacent intervals of the same generation. The function $S(\cdot,\ba)$ may fail to be globally continuous on $\Pcal_*(\bc)$ if
the interval crosses the discontinuity of the inverse branches. However, the
restriction of $S(\cdot,\ba)$ to $\Pcal(\bc)$ has a natural $C^r$ extension to
$\Pcal_*(\bc)$. We denote it by
\[
    S_{\bc}(\cdot,\ba):\Pcal_*(\bc)\to\R^d.
\]
Let us fix a positive number $\kappa$ and consider the mapping $T$
for a function $\phi$ in
\[
    \mathcal{U}_\kappa=
    \left\{
    \phi\in C^r(\Sone,\R^d);
    \ \|\phi\|_{C^r}
    :=
    \max_{0\leq k\leq r}
    \sup_{x\in \Sone}
    \norm{\frac{d^k}{dx^k}\phi(x)}
    \leq \kappa
    \right\}.
\]
Since $\rho(A)<1$, we can define
\begin{equation}
\label{eq:alpha}
    \alpha=\kappa\displaystyle\sum_{i=0}^{\infty}\norm{A^i} <+\infty \qquad \text{and} \qquad D=\Sone\times\{y\in\R^d:\ \norm{y}\leq\alpha\}.
\end{equation}
For the same reason there exists an integer $N>0$, for which $T^N(D)\subseteq D$. We can write the attractor of $T$ as
$$\Lambda^\phi=\bigcap_{n=0}^\infty T^{nN}(D).$$ Moreover,
for any word $\ba$ of finite or infinite length, a straightforward computation shows that
\[
    \sup_{x\in P^*(\bc)}
    \max_{0\leq k\leq r}
    \ell^k
\norm{\frac{d^{k}}{dx^{k}}S_\bc(x,\ba)}
    \leq \alpha.
\]

\subsection{Derivative Dispersion and Concentration Multiplicity}

The purpose of this section is to quantify the  concentration
of the derivatives of the graph functions at a natural dynamical
scale. For a fixed inverse cylinder, we count how many
$q$-branches may have their derivatives contained in the same thin
slab. The resulting quantity will be called the \emph{concentration
multiplicity}. The \emph{derivative dispersion condition} requires this
multiplicity to grow subexponentially.

\begin{definition}
For $\bc\in\A^p$ and $(\xi,\eta)\in\Z\times\R^d$ we define the set
\begin{equation*}
B_\bc^q(\xi,\eta)
:=
\left\{
 \ba\in\A^q:
 \exists x\in \Pcal_*(\bc)\ \text{such that}\
 \left|\xi+\eta\cdot S'_\bc(x,\ba)\right|
 < 2\alpha{\ell}^{-q} \norm{(A^q)^T\eta}
\right\}
\end{equation*}
and we set
\begin{equation*}
G_\bc^q(\xi,\eta):=\A^q\setminus B^q_\bc(\xi,\eta).
\end{equation*}
\end{definition}

\begin{definition}
For \(p,q\geq1\), we define the finite scale concentration multiplicity as
\[
    M(p,q)
:=
\max_{c\in\A^p}
\sup_{(\xi,\eta)}
\#
B^q_\bc(\xi,\eta).
\]
\end{definition}

\begin{remark} Equivalently, defining the dynamical $q$-scale slab orthogonal to $\be\in\mathbb{S}^{d-1}$ as 
\begin{equation*}
    E_q(\be,r):=\left\{y \in \R^d : |\be \cdot y-r| < 2 \alpha \ell^{-q}\norm{(A^q)^T\be}\right\},
\end{equation*}
we get
\begin{equation*}
M(p,q)
:=
\max_{\bc\in\A^p}
\sup_{\be\in \mathbb{S}^{d-1}}
\sup_{r\in\R}
\#
\left\{
 \ba\in\A^q:
 \exists x\in \mathcal{P}^*(\bc)\ \text{s.t.}\
 S'_\bc(x,\ba) \in E_q(\be,r) \right\}
\end{equation*}
\end{remark}

The intervals \(\Pcal_*(\bc)\) shrink as \(p\) increases, so the \emph{finite scale concentration multiplicity} \(M(p,q)\) is non-increasing in \(p\). 

\begin{definition}[Concentration Multiplicity]
We define the
concentration multiplicity at time \(q\) by
\[
    M(q)
    :=
    \lim_{p\to\infty}M(p,q)
    =
    \inf_{p\geq1}M(p,q).
\]
\end{definition}
We are finally ready to introduce the derivative dispersion condition.
\begin{definition}[Derivative dispersion]
We say that $\phi$ satisfies the derivative dispersion condition with exponent $\beta>0$, denoted by $(\bD_\beta),$ if
\begin{equation*}
\limsup_{q\to\infty}
\frac{1}{q}\log M(q)
<
\beta.
\end{equation*}
We say that $\phi$ satisfies the derivative dispersion condition $(\bD)$ if it fulfills $(\bD_\beta)$ for any $\beta>0.$
\end{definition}

\subsection{The Anisotropic Sobolev Norms}
\label{sec:Anisotropic_Sobolev_Norm}
In this section we introduce a family of anisotropic Sobolev spaces adapted to the spectral
decomposition of the matrix $A$.

Let
$
1>|\lambda_1|\geq\cdots\geq|\lambda_t|>0
$
be the eigenvalues of $A$ corresponding to the distinct Jordan blocks, ordered in non-increasing order.
We consider the real $A^T$-invariant spectral decomposition
\[
\mathbb R^d=E_1\oplus\cdots\oplus E_t,
\]
where each $E_i$ is either the generalized eigenspace associated with a real
eigenvalue of $A$, or the real invariant generalized eigenspace associated
with a pair of complex conjugate eigenvalues. We set
$
B_i:=A^T|_{E_i},
$ and $
d_i:=\dim E_i.
$
Let
$
\pi_i:\mathbb R^d\to E_i
$
be the projection associated with this decomposition and, for
$\eta\in\mathbb R^d$, we write
$
\eta_i:=\pi_i(\eta).
$

Given a vector
$
\mathbf s=(s_1,\ldots,s_t)\in\mathbb R^t,
$ such that $
s_i>-\frac{d_i}{2},
$
and, for $u\in C_c^\infty(\mathbb S^1\times\mathbb R^d)$, we define the
homogeneous anisotropic Sobolev norm by
\[
\|u\|_{L^2(\mathbb S^1,\dot H_A^{\mathbf s}(\R^d))}^2
:=
\int_{\mathbb S^1}\int_{\mathbb R^d}
|\mathcal F_yu(x,\eta)|^2
\prod_{i=1}^t\norm{\eta_i}^{2s_i}
\,d\eta\,dx,
\]
and the non-homogeneous anisotropic Sobolev norm by
\[
\|u\|_{L^2(\mathbb S^1,H_A^{\mathbf s}(\R^d))}^2
:=
\int_{\mathbb S^1}\int_{\mathbb R^d}
|\mathcal F_yu(x,\eta)|^2
\prod_{i=1}^t\langle\eta_i\rangle^{2s_i}
\,d\eta\,dx,
\]
where
$
\langle\eta_i\rangle:=(1+\norm{\eta_i}^2)^{1/2}.
$
\\ 

Equivalently, by Plancherel's theorem in the variable $x$,
\[
\|u\|_{L^2(\mathbb S^1,\dot H_A^{\mathbf s}(\R^d))}^2
=
\sum_{\xi\in\mathbb Z}
\int_{\mathbb R^d}
|\widehat u(\xi,\eta)|^2
\prod_{i=1}^t\norm{\eta_i}^{2s_i}
\,d\eta,
\]
and
\[
\|u\|_{L^2(\mathbb S^1,H_A^{\mathbf s}(\R^d))}^2
=
\sum_{\xi\in\mathbb Z}
\int_{\mathbb R^d}
|\widehat u(\xi,\eta)|^2
\prod_{i=1}^t\langle\eta_i\rangle^{2s_i}
\,d\eta.
\]

We denote by
$
L^2(\mathbb S^1,\dot H_A^{\mathbf s}(\R^d))
$ and $
L^2(\mathbb S^1,H_A^{\mathbf s}(\R^d))
$ the completions of
$C_c^\infty(\mathbb S^1\times\mathbb R^d)$ with respect to $\|\cdot\|_{L^2(\mathbb S^1,\dot H_A^{\mathbf s}(\R^d))}^2$  and $\|\cdot\|_{L^2(\mathbb S^1,H_A^{\mathbf s}(\R^d))}^2,$ respectively.

\begin{remark}
The condition
$
s_i>-\frac{d_i}{2},
$ sor every $ i=1,\ldots,t,
$
ensures that the homogeneous weight is locally integrable on $\mathbb R^d$, and hence that the homogeneous norm is
finite on $C_c^\infty(\mathbb S^1\times\mathbb R^d)$.
\end{remark}

\subsection{Main Results}
The first result we state reduces anisotropic Sobolev regularity of the SRB measure to a bound on the concentration multiplicity.

\begin{theorem}
\label{thm:1}
    Let $\bs$ be a vector in $\mathbb{R}^t$. Setting $k = 2(d-1)\sum_{i=1}^t |s_i|$, there exists a positive constant $B_\bs=B_\bs(A)$, depending only on $A$ and $\bs,$ such that the following holds. If $\phi\in C^r(\Sone,\mathbb{R}^d),$ with regularity $r> \sum_{i=1}^t s_i + (d+3)/2$, and there exists an integer $q\geq N$ such that
    \begin{equation}
    \label{eq:thm1}
        B_\bs q^kM(q)<\left( \ell |\det(A)|\prod_{i=1}^t |\lambda_i|^{2s_i} \right)^q.
    \end{equation}
    Then $\mu_{\phi} \in L^2(\Sone,\dot H^\bs_A(\R^d))$
\end{theorem}

\begin{remark}
\label{rem:monotonicity}
The constant $B_\bs$ is explicitly defined in \eqref{eq:constant} in terms of the norms of the matrices $B_i=A^T|_{E_i}.$ In particular, $B_\bs$ is increasingly monotonic on any half-line $s_i\in[0,+\infty).$ Thus, if \eqref{eq:thm1} is true for $s_i\geq0,$ then it is also satisfied for $\bs=\mathbf 0$ and Theorem \ref{thm:1} gives $\mu_\phi\in L^2(\Sone,\dot H_A^\mathbf 0(\R^d)).$
\end{remark}

Secondly, we fully formulate the measure theoretic prevalence of the derivative dispersion condition.\newline
For $\beta>0$, we set

\begin{equation*}
    \mathcal{E}(\beta,\kappa,A) = \left\{ \phi\in\mathcal{U}_\kappa : \limsup_{q\rightarrow\infty} \frac 1q \log M(q) \geq \beta \right\}.
\end{equation*}
\begin{remark}\label{M(q) open condition}
For each fixed \(q\), the quantity $M(q)$ varies upper semi-continuously in $\phi$. It follows that the complement of
\(\mathcal E(\beta,\kappa,A)\) is open, for this reason Theorem \ref{thm:C} can refer to the interior of $\mathcal{F}$.
\end{remark}

\begin{theorem}\label{genericity}
    For every $\beta>0$, $\ell\geq2$ and for $A \in GL_d(\R)$ with $\rho(A)<1$, there exist a family of smooth functions $\psi_1,...,\psi_m \in C^{\infty}(\Sone, \R^d)$ and $D_0>0$ such that, for every $\kappa>D_0$ and $\phi \in \mathcal{U}_{\kappa-D_0}$,
    \begin{equation*}
        \textit{Leb}\left( \left\{(t_1,...,t_m) \in [-1,1]^m : \phi+\sum_{i=1}^{m} t_i\psi_i \in \mathcal{E}(\beta,\kappa,A)\right\} \right)=0.
    \end{equation*}
\end{theorem}

We finally record the geometric consequences of anisotropic Sobolev
regularity.

\begin{theorem}
\label{thm:3}
    We have the following.
    \begin{enumerate}
        \item \label{it:thm1_item1} Let $\bs\in\mathbb{R}^t$ satisfy $-d_i/2<s_i<0$ for every $i$ and suppose that $\mu\in L^2(\Sone, \dot H^\bs_A(\R^d)),$ with compact support. Then, calling $\mu_x$ the conditional measures of $\mu$ along the fibers, for a.e. $x \in\Sone$,
        \[
            \text{dim}_H \left(\mu_{x}\right) \geq d+2(s_1+\ldots +s_t).
        \]
        
        \item Let $\bs\in\mathbb{R}^t$ satisfy $s_i \geq 0$ for every $i$ and suppose that $\mu\in L^2(\Sone, \dot H^\bs_A(\R^d)),$ with compact support. Then $\mu$ is absolutely continuous with respect to the Lebesgue measure on $\Sone \times \R^d$.\label{it:thm1_item2} 
        
        \item Let $\mathcal{S}$ be a finite collection of vectors in $\mathbb{R}^t$ with $s_i\geq 0$ for every $i$ and every $\bs\in\mathcal{S}$. Suppose that $\mu\in L^2(\Sone, \dot H^{\bs}_A(\R^d))$ for every $\bs\in \mathcal{S}$, and it is compactly supported. Then
        \[
        \mu\in L^2(\Sone, H^{s}(\R^d)),\qquad \text{ where } \qquad s = \min_{1\leq i \leq t} \max_{\bs\in \mathcal{S}} s_i. \label{it:thm1_item3} 
        \]
    \end{enumerate}
\end{theorem}

Mere fiberwise Sobolev regularity of the SRB measure can be turned into usual isotropic Sobolev regularity, exploiting the regularity of the unstable leaves, as the next lemma shows.

\begin{proposition}
\label{prop:Proposition_regularity}
    Suppose that the SRB measure $\mu_{\phi}$ belongs to $L^2(\Sone, H^{s}(\R^d))$ for an $s>0$, such that $|\det(A)|\rho(A^{-1})^{-2(s-1)}>1/\ell^3$. Then $\mu_{\phi}$ is absolutely continuous and its density $\Psi_{\phi}$ belongs $H^s(\Sone\times\mathbb{R}^d)$.
\end{proposition}

\section{Proof of the Statements in the Introduction}\label{section: proof of intro}

In this section we prove the assertions made in the introduction, assuming the general theorems stated in section \ref{section: setting}. Note that Theorem \ref{thm:C} is a direct consequence of Theorem \ref{genericity} and Remark \ref{M(q) open condition}.

\subsection{Proof of Proposition \ref{upper bound on dim h}}

Proposition \ref{upper bound on dim h} follows from the argument introduced by Falconer in \cite{falconer1988hausdorff}, where he proved upper bounds for the Hausdorff dimension of self-affine fractal sets.

\begin{proof}[Proof of Proposition \ref{upper bound on dim h}]
If $\dim_L(A,\ell)=d$, the assertion is immediate. Otherwise fix $s\in(\dim_L(A,\ell),d)$, set $k=\lfloor s\rfloor$, and let
$\alpha_1(n)\geq\cdots\geq\alpha_d(n)$ be the singular values of $A^n$. Recall the constant $\alpha$, introduced in \eqref{eq:alpha}, fix a radius $R>\alpha$ and denote by $B(0,R)$ the ball centered at zero of radius $R$. The set $\Lambda^{\phi}_x$ is contained in the union of $\ell^n$ translates of
$A^nB(0,R)$, which are ellipsoids with semiaxes
$R\alpha_1(n),\ldots,R\alpha_d(n)$. It is possible to cover every ellipsoid by cubes of size
$\delta_n:=R\alpha_{k+1}(n)$.  Up to a dimensional constant, the number of cubes
needed is
$\prod_{i=1}^k\alpha_i(n)/\alpha_{k+1}(n)$.  Summing over the $\ell^n$
ellipsoids, the $s$-dimensional $CR\alpha_{k+1}(n)$-Hausdorff measure 
\begin{equation}\label{eq:content}
  \mathcal{H}^s_{C R\alpha_{k+1}(n)}(\Lambda^\phi_x)
  \leq C R^s\ell^n
  \alpha_1(n)\cdots\alpha_k(n)\alpha_{k+1}(n)^{s-k},
\end{equation}
where $C$ is independent of $n$ and $x$.

It remains to compute the exponential rate of the last factor.  For
$j=1,\ldots,d$, one has
$\alpha_1(n)\cdots\alpha_j(n)=\lVert\bigwedge^j A^n\rVert$.  Hence the
spectral-radius formula, applied to $\bigwedge^j A$, yields
\[
  \lim_{n\to\infty}
  \bigl(\alpha_1(n)\cdots\alpha_j(n)\bigr)^{1/n}
  =\rho_1\cdots\rho_j.
\]
Writing $s=k+\theta$, the singular-value factor in
\eqref{eq:content} is the geometric interpolation of the cases $k$ and
$k+1$.  Its $n$-th root therefore tends to
$\rho_1\cdots\rho_k\rho_{k+1}^{\theta}=\Phi(s)$.  Since
$s>\dim_L(A,\ell)$, we have $\ell\Phi(s)<1$, so the right-hand side of
\eqref{eq:content} tends to zero exponentially.  Moreover
$\delta_n\leq R\lVert A^n\rVert\to0$.  Thus
$\mathcal{H}^s(\Lambda^\phi_x)=0$, and consequently $\dim_H (\Lambda^\phi_x)\leq s$.  We now move from the vertical sections to the full attractor. Since the graphs are Lipschitz continuous, it follows that, there exists $L>0$ such that, 
\[
    d_H(\Lambda_x^\phi,\Lambda_{x'}^\phi)
    \leq
    Ld_{\Sone}(x,x'),
\]
where $d_H$ denotes the Hausdorff distance on the space of the compact subsets of $\mathbb{R}^d$.\newline
We consider a partition of $\Sone$ into intervals \(I_1,\ldots,I_{m_n}\) of length at
most \(\delta_n\), where $m_n\leq 2\delta_n^{-1},$
and choose \(x_j\in I_j\). For every $x \in I_j$,
\[
\Lambda_x^\phi
\subset
\left\{
y\in\mathbb R^d:
\operatorname{dist}\bigl(y,\Lambda_{x_j}^\phi\bigr)
\leq L\delta_n
\right\}.
\]
The fibrewise construction above provides a cover of
\(\Lambda_{x_j}^\phi\) by at most $N_n
    \leq
    C\ell^n
    \prod_{i=1}^k
    \frac{\alpha_i(n)}{\alpha_{k+1}(n)}
$
sets of diameter at most \(C\delta_n\), with the constant independent
of \(j\). Enlarging these sets by \(L\delta_n\), and taking their product
with \(I_j\), gives a cover of
$\Lambda^\phi\cap
    \bigl(I_j\times\mathbb R^d\bigr)$
by at most \(N_n\) sets of diameter at most \(C_1\delta_n\), with $C_1$ independent from $x$ and $n$. Summing over the intervals of the partition, we obtain
\[
\begin{aligned}
    \mathcal H^{s+1}_{C_1\delta_n}(\Lambda^\phi)
    &\leq
    2C_1^{s+1}\delta_n^{-1}N_n\delta_n^{s+1}
    =
    2C_1^{s+1}N_n\delta_n^s\\
    &\leq
    2CC_1^{s+1}R^s\ell^n
    \alpha_1(n)\cdots\alpha_k(n)
    \alpha_{k+1}(n)^{s-k},
\end{aligned}
\]

The last expression differs from the right hand side of \eqref{eq:content} only by a multiplicative constant independent of $n$. Since
\(C_1\delta_n\to0\), it follows that
$\mathcal H^{s+1}(\Lambda^\phi)=0,$
and consequently
$\dim_H(\Lambda^\phi)\leq s+1.$
Finally, letting \(s\downarrow\dim_L(A,\ell)\), we conclude that $\dim_H(\Lambda_x^\phi)
    \leq
    \dim_L(A,\ell)$ for every \(x\in\Sone\), and
\[
    \dim_H(\Lambda^\phi)
    \leq
    1+\dim_L(A,\ell).
\]
\end{proof}

\subsection{Proof of Theorems \ref{thm:A} and \ref{thm:B}}

Theorems \ref{thm:A} and \ref{thm:B} follow from a combination of Theorems \ref{thm:1} and \ref{thm:3}. For Theorem \ref{thm:A}, it is necessary to optimize the choice of $\bs=(s_i,\dots,s_i)$. We start by recalling an equivalent definition of the affinity dimension (also called Lyapunov dimension, when referred to measures), that facilitate this optimization procedure.

\begin{definition}[Fiber Lyapunov dimension]
Fix $A\in GL_d(\R)$ such that $\rho(A)<1$, and
$\ell\geq 2$. Let
\begin{equation*}
1>|\lambda_1|\geq\cdots\geq|\lambda_t|>0
\end{equation*}
be as above the eigenvalues of $A$ corresponding to the Jordan block decomposition, ordered in non-increasing
order, and let $d_i$ be the real dimension of the generalized spectral
subspace associated with the eigenvalues of modulus $|\lambda_i|$.
The fiber Lyapunov dimension of $(A,\ell)$ is defined by
\begin{equation}
\label{eq:affinity_dimension_2}
\dim_{\mathrm L}(A,\ell)
:=
\sup \left\{
\sum_{i=1}^t\beta_i:
0\leq\beta_i\leq d_i,\quad
-\sum_{i=1}^t\log|\lambda_i|\beta_i\leq\log\ell
\right\}.
\end{equation}
\end{definition}

\begin{remark}
 Assume that $\ell|\det A|<1$, and let
$k\in{1,\ldots,t}$ be the unique index such that
\begin{equation*}
-\sum_{i=1}^{k-1}d_i\log|\lambda_i|
\leq\log\ell
<
-\sum_{i=1}^{k}d_i\log|\lambda_i|.
\end{equation*}
Then
\begin{equation*}
\dim_{\mathrm L}(A,\ell)
=
\sum_{i=1}^{k-1}d_i
-
\frac{
\displaystyle
\log\ell+
\sum_{i=1}^{k-1}d_i\log|\lambda_i|
}{
\displaystyle
\log|\lambda_k|
}.
\end{equation*}
\end{remark}

\begin{remark}
For $\mathbf s=(s_1,\ldots,s_t)$, set $\beta_i:=d_i+2s_i$.
Then $
\ell\prod_{i=1}^t|\lambda_i|^{d_i+2s_i}>1
$
if and only if $
-\sum_{i=1}^t\log|\lambda_i|\beta_i<\log\ell.
$
Thus the admissible Sobolev exponents correspond precisely to the
strictly admissible vectors in the optimization problem defining
$\dim_{\mathrm L}(A,\ell)$ in \eqref{eq:affinity_dimension_2}.
\end{remark}

\begin{proof}[Proof of Theorem~\ref{thm:A}]
The upper bound follows from Proposition \ref{upper bound on dim h}. Indeed, we have that $\supp(\mu_\phi)= \Lambda^\phi$ and that $\supp(\mu_{\phi,x})\subset \Lambda^\phi_x$ for a.e. $x$, implying that $\DimH(\mu_\phi)\leq \DimH(\Lambda^\phi)$ and $\DimH(\mu_{\phi,x})\leq \DimH(\Lambda^\phi_x)$ for a.e. $x\in\Sone$.\newline

We can now focus on the lower bound. 
Let $\boldsymbol\beta=(\beta_1,\ldots,\beta_t)$
satisfy
\[
    0<\beta_i<d_i
    \qquad\text{and}\qquad
    -\sum_{i=1}^t\log(|\lambda_i|)\beta_i<\log\ell.
\]
Set $s_i:=\frac{\beta_i-d_i}{2}$,
then $-d_i/2<s_i<0$. 
The $(\bD)$-condition gives \eqref{eq:thm1} for some $q\geq 1.$ Consequently, 
by Theorem~\ref{thm:1} point \emph{\ref{it:thm1_item1}.}~we have $\mu_\phi\in L^2\bigl(\Sone,\dot H_A^{\mathbf s}\bigr)$ and by Theorem~\ref{thm:3}, we get that, for almost every
$x\in\mathbb S^1$,
\[
    \dim_H(\mu_{\phi,x})
    \geq
    d+2\sum_{i=1}^ts_i
    =
    \sum_{i=1}^t\beta_i.
\]
By \cite[Theorem~7.7]{Mattila1995}, applied to the projection onto
the first coordinate, it follows that
$\DimH(\mu_\phi)
    \geq
    1+\sum_{i=1}^t\beta_i$,
for every admissible $\boldsymbol\beta$. Taking the supremum over all admissible $\boldsymbol\beta$, we obtain
\begin{equation*}
    \dim_H(\mu_\phi)\geq 1+\dim_{\mathrm L}(A,\ell),
\end{equation*}
which proves the lower bound.\newline

\end{proof}

\begin{proof}[Proof of Theorem~\ref{thm:B}]
    The fact that $|\det(A)|\ell >1$ and the $(\bD)$-condition imply that the hypotheses of Theorem \ref{thm:1} hold for $\bs=0$. We get that $\mu_\phi \in L^2\bigl(\mathbb S^1,\dot H_A^{\mathbf 0}\bigr)$. Point \emph{\ref{it:thm1_item2}.}~of Theorem~\ref{thm:3} implies then that $\mu_\phi$ is absolutely continuous with respect to the Lebesgue measure.\newline
    For the second part of the statement, fix $s>0$ and assume that $|\det(A)|\rho(A^{-1})^{-2s} >1/\ell$. Let $\mathcal{S}$ be the collection of the $t$ vectors in $\mathbb{R}^t$ which have exactly one coordinate equal to $s$ and the remaining ones equal to $0$. The $(\bD)$-condition and our initial assumption imply that the hypothesis of Theorem \ref{thm:1} hold for any $\bs\in\mathcal{S}$. Finally, point \emph{\ref{it:thm1_item3}.}~of Theorem~\ref{thm:3} and the assumed inequality enable us to apply Proposition \ref{prop:Proposition_regularity} and to conclude the proof.
    
\end{proof}

\subsection{Proof of Example \ref{example times 3}}

Finally, we prove example \ref{example times 3}. The structure of this map is particularly suited to produce dispersion of the derivatives of the graphs $S(x,\ba)$.

\begin{proof}[Proof of Example \ref{example times 3}]

We identify $\Sone\times\R^2$ with $\Sone\times \C$, hence the map $T$ can be written as $T(x,w)=\big(3x, \lambda w + e^{2\pi i x}\big)$. We define the non-dynamical slab orthogonal to $\be\in \Sone$ of position $r\in\R$ and width $t>0$ as $E(\be,r,t)=\left\{y \in \R^2 : |\be \cdot y-r| < t\right\}$. The assertion in the example will follow from the next proposition.

\begin{proposition}
\label{prop:example}
For every \(0<\lambda<1\) and \(C_1>0\), there is a constant
\(C_2=C_2(\lambda,C_1)\) such that, for every \(x\), every \(q\geq1\), every $\be\in \Sone$ and $r\in \R$,
\[
 \#\left\{\mathbf a\in\A^q:
 \frac{d}{dx}S(x,\mathbf a)\in
 E\left(\be,r,C_1\left(\frac{\lambda}3\right)^q\right)\right\}
 \leq C_2\,2^q .
\]
\end{proposition}

\begin{proof}[Proof of Proposition~\ref{prop:example}]
Put \(R=\lambda/3\) and write
\(D(\mathbf a)=\frac1{2\pi}\frac{d}{dx}S(x,\mathbf a)\).  If
\(\mathbf a\in\A^k\) and \(\mathbf a j\in \A^{k+1}\) is obtained by appending
\(j\in\A\), differentiation of the last term in \(S(x,\mathbf a j)\)
gives the recurrence
\begin{equation}\label{eq:recurrence}
 D(\mathbf a j)
 =D(\mathbf a)-\frac{i}{3}R^k
   e^{-2\pi i\big(\mathbf a j(x)\big)}.
\end{equation}
Consequently the three direct descendants of \(D(\mathbf a)\) are the vertices
of an equilateral triangle centred at \(D(\mathbf a)\), with
circumradius \(R^k/3\). The recurrence also shows that a descendant \(\mathbf b\in\A^m\) of
\(\mathbf a\in\A^k\), \(m>k\), satisfies
\begin{equation}\label{eq:tail-and-constant}
 |D(\mathbf b)-D(\mathbf a)|\leq AR^k,
 \qquad
 A=\frac{1}{3(1-R)}.
\end{equation}
Choose \(c>0\) so that \((A+c)R<1/4\); for instance one may take
\[
 c=c_\lambda
 :=\frac12\left(\frac{1}{4R}-A\right)
 =\frac{9-7\lambda}{8\lambda(3-\lambda)}.
\]

Fix a direction $\be\in \Sone$ and a coordinate $r\in\R$. Call \(\mathbf a\in\A^k\) surviving if
\(D(\ba)\in E\big(\be,r,(A+c)R^k\big)\).  Estimate
\eqref{eq:tail-and-constant} has two immediate consequences.  A word
which is not surviving has no descendant $\bb\in\A^m$ in
\( E(L,cR^m)\); conversely, if
\(D(\mathbf b)\in E(\be,r,cR^q)\), then every prefix of \(\mathbf b\)
is surviving.

For every surviving word we bound from above how many of its direct descendant can be surviving.  The orthogonal
projection of a unit equilateral triple onto any line has range at
least \(3/2\).  Hence , given $\ba\in\A^k$, by \eqref{eq:recurrence}, the three orthogonal projections of $D(\ba j)$, for $j=0,1,2$, have range at least \(R^k/2\).
At least one of them therefore has absolute value at least \(R^k/4\),
which is strictly larger than \((A+c)R^{k+1}\).  Thus every surviving
word has at most two surviving direct descendants.  Starting from the empty
word, this gives
\[
 \#\{\mathbf a\in\A^q:
 D(\mathbf a)\in E(\be,r ,cR^q)\}\leq2^q.
\]

Finally, a slab of width \(C_1R^q\) is covered by
\(\max\{1,\lceil C_1/c\rceil\}\) parallel slabs of width \(cR^q\).  The preceding estimate applied to these
parallel axes proves the claim, with \(C_2=\max\{1,\lceil C_1/c_\lambda\rceil\}\).
\end{proof}

Fix $q>0$. By the uniform continuity of $\frac d{dx}S(\ba,x)$, we can pick $p(q)>0$ so large that for, any $\bc\in\A^{p(q)}$, $\frac d{dx}S_\bc(\ba,x)\in E(\be,r,2\alpha(\lambda/3)^q)$ for some $x\in\mathcal{P}^*(\bc)$ implies that $\frac d{dx}S_\bc(\ba,x_\bc)\in E(\be,r,4\alpha(\lambda/3)^q)$, where $x_\bc$ denotes the left endpoint of $\mathcal{P}^*(\bc)$. Applying the previous Lemma for $q$ and $t=4\alpha(\lambda/3)^q$ we obtain that
\begin{equation*}
    M(q)\leq C_2 2^q.
\end{equation*}
Recalling that $3\lambda^2>2$ and provided that $q$ is sufficiently large, the hypothesis of Theorem \ref{thm:1} are satisfied for $\bs=0$. In combination with Theorem \ref{thm:3}, this implies that the SRB measure is absolutely continuous, as wished.

\end{proof}

\section{Anisotropic Sobolev Norms and Energies}\label{section: energy and sobolev}

In this section we collect some properties of the anisotropic Sobolev norms and use them to prove Theorem \ref{thm:3}. We begin with the comparison between the homogeneous and non-homogeneous norms for compactly supported functions. We then introduce the anisotropic energy associated with the matrix \(A\) and we relate it to the corresponding Sobolev norm. Using the classical energy criterion, this yields the lower bound for the Hausdorff dimension required in Theorem \ref{thm:3}.

\subsection{Compact support estimates} 

The decomposition \(\mathbb R^d=E_1\oplus...\oplus E_t\), introduced in Section \ref{sec:Anisotropic_Sobolev_Norm}, describes the \(A^{T}\)-invariant subspaces in the frequency space. We now introduce the corresponding dual directions in the physical space. Set $F_i := \left( \bigoplus_{j\neq i}E_j \right)^\perp$, then $\dim F_i=d_i,$ and $\mathbb R^d=F_1\oplus\cdots\oplus F_t$ is an $A-$invariant decomposition. We denote by $\pi'_i:\R^d \rightarrow F_i$ the projections associated with this decomposition.
\begin{lemma}
\label{lemma:relation_Sobolev_norms}
Let
$s_i>0$ and let
\(K\subset\mathbb R^d\) be a compact subset. Then there exists
\(C=C(A,K,\mathbf s)>0\) such that, for every
\(u\in C_c^\infty(\mathbb R^d)\) satisfying
\(
\operatorname{supp}u\subset K,
\)
one has
\[
\|u\|_{H_A^{\mathbf s}(\R^d)}
\leq C\,\|u\|_{\dot H_A^{\mathbf s}(\R^d)}.
\]
\end{lemma}

\begin{proof}
Write
$y=y_1+\cdots+y_t$
with $y_i\in F_i$
and $\eta=\eta_1+\cdots+\eta_t$,
with $\eta_i\in E_i.$
By the definition of \(F_i\),
\[
\langle y,\eta\rangle
=
\sum_{i=1}^t\langle y_i,\eta_i\rangle.
\]
For each \(i\), the Euclidean pairing identifies \(E_i\) linearly
with \(F_i^*\). Since this identification is fixed, the Euclidean
norms on \(E_i\) and \(F_i^*\) are equivalent. By the fractional Poincar\'e inequality
\cite[Lemma~5.4]{RailoZimmermann2023}, there exists
\(C_i=C_i(A,K,s_i)>0\) such that
\[
\int_{E_i}
\langle\eta_i\rangle^{2s_i}
|\widehat f(\eta_i)|^2\,d\eta_i
\leq
C_i
\int_{E_i}
\norm{\eta_i}^{2s_i}
|\widehat f(\eta_i)|^2\,d\eta_i
\]
for every function \(f\) on \(F_i\) supported in the fixed compact
set \(\pi_i'(K)\). \\
\noindent
We now apply this estimate on every $E_i$. Fix \(i\), and \(\eta_j\in E_j\) for every \(j\neq i\).
Define the auxiliary function \(g_i:F_i\to\mathbb C\) by
\[
g_i(y_i)_{\hat{\eta_i}}
:=
\int_{\prod_{j\neq i}F_j}
u\left(\sum_{j=1}^t y_j\right)
\exp\left(
-2\pi i
\sum_{j\neq i}\langle y_j,\eta_j\rangle
\right)
\prod_{j\neq i}dy_j.
\]
Since \(u\) is supported in \(K\),
\[
\operatorname{supp}g_i\subset \pi_i'(K),
\]
uniformly in \(x\) and in the fixed frequencies \(\eta_j\),
\(j\neq i\). \\
\noindent 
Moreover, up to the constant Jacobian associated with the fixed
decomposition
\(\mathbb R^d=F_1\oplus\cdots\oplus F_t\),
\[
\widehat u\left(\sum_{j=1}^t\eta_j\right)
=
C_A
\int_{F_i}
g_i(y_i)e^{-2\pi i\langle y_i,\eta_i\rangle}\,dy_i.
\]
Thus, as a function of \(\eta_i\), the full Fourier transform of \(u\)
is, up to the fixed factor \(C_A\), the Fourier transform of \(g_i\).
Applying the previous estimate to \(g_i\) gives
\[
\int_{E_i}
\langle\eta_i\rangle^{2s_i}
\left|
\widehat u\left(\sum_{j=1}^t\eta_j\right)
\right|^2\,d\eta_i
\leq
C_i
\int_{E_i}
\norm{\eta_i}^{2s_i}
\left|
\widehat u\left(\sum_{j=1}^t\eta_j\right)
\right|^2\,d\eta_i.
\]
\noindent
Starting with \(i=1\), we multiply this inequality by
\(
\prod_{j=2}^r\langle\eta_j\rangle^{2s_j}
\)
and in the remaining frequency variables.
Fubini's theorem yields
\[
\begin{aligned}
&\int_{\mathbb R^d}
|\widehat u(\eta)|^2
\prod_{j=1}^t\langle\eta_j\rangle^{2s_j}
\,d\eta
\leq
C_1
\int_{\mathbb R^d}
|\widehat u(\eta)|^2
\norm{\eta_1}^{2s_1}
\prod_{j=2}^t\langle\eta_j\rangle^{2s_j}
\,d\eta.
\end{aligned}
\]
Repeating the same argument successively for
\(i=2,\ldots,t\), we obtain
\[
\int_{\mathbb R^d}
|\widehat u(\eta)|^2
\prod_{i=1}^t\langle\eta_i\rangle^{2s_i}
\,d\eta
\leq
\prod_{i=1}^{t}C_i
\int_{\mathbb R^d}
|\widehat u(\eta)|^2
\prod_{i=1}^t\norm{\eta_i}^{2s_i}
\,d\eta,
\]
which is the desired estimate.
\end{proof}

\begin{corollary}
    If $\mu \in L^2(\mathbb{S}^1, H_A^{\mathbf s})$ then $\mu \in L^2(\mathbb{S}^1, H^s)$ with $s=\min_{i=1,..,t} s_i.$
\end{corollary}

\subsection{Anisotropic Energy}
Assume that $-d/2<s<0.$ For a finite compactly supported Borel measure $\mu$ on $\mathbb R^d$, recall that its $(d+2s)-$energy is defined by 
\[ I_{d+2s}(\mu) := \int_{\R^d}\int_{\R^d} \frac{d\mu(y)\,d\mu(z)} {\norm{y-z}^{d+2s}}. \] 

Since $0<d+2s<d$, the Fourier representation of the energy
\cite[Lemma~12.12]{Mattila1995} gives \[ I_{d+2s}(\mu) = c_{d,s} \int_{\mathbb R^d} |\widehat\mu(\eta)|^2 \norm{\eta}^{2s} \,d\eta, \] where $c_{d,s}>0$ depends only on $d$, $s$, and on the convention used to compute the Fourier transform. 

\begin{definition}
    Let $\bs=(s_1,...,s_t) \in \R^t$ be such that $-d_i/2<s_i<0.$ For a finite compactly supported Borel measure $\mu$ on $\mathbb R^d$, we define its $(A,\mathbf s)-$energy by
\[ I_{A,\mathbf s}(\mu) := \int_{\R^d} \int_{\R^d}\prod_{i=1}^t \frac{1} {\norm{\pi'_i(y-z)}^{d_i+2s_i}} \,d\mu(y)\,d\mu(z). \]
\end{definition} 
Since $0<d_i+2s_i<d_i$ applying the Fourier representation of the Riesz energy \cite[Lemma~12.12]{Mattila1995} on each factor gives \[ I_{A,\mathbf s}(\mu) = c_{A,\mathbf s} \int_{\mathbb R^d} |\widehat\mu(\eta)|^2 \prod_{i=1}^t \norm{\eta_i}^{2s_i} \,d\eta, \] where $c_{A,\mathbf s}>0$ depends on the decomposition, on $\mathbf s$ and on the convention used to compute the Fourier transform. Consequently, if \[ \nu=\int_{\mathbb S^1}\nu_x\,dx, \] then \[ \int_{\mathbb S^1} I_{A,\mathbf s}(\nu_x) \,dx = c_{A,\mathbf s} \int_{\mathbb S^1} \int_{\mathbb R^d} |\widehat{\nu_x}(\eta)|^2 \prod_{i=1}^t \norm{\eta_i}^{2s_i} \,d\eta\,dx = c_{A,\mathbf s}\|\nu\|_{L^2(\Sone,\dot H_A^\bs(\R^d))}^2. \]

We now compare the $(A,\mathbf s)$-energy with the classical energy. Given $
    S=\sum_{i=1}^t s_i,
$ since the projections $\pi_i:\mathbb R^d\rightarrow F_i$ are continuous and $s_i<0$, there exists a constant $C_{A,\mathbf s}>0$ such that $ \norm{\eta}^{2S} \leq C_{A,\mathbf s} \prod_{i=1}^t \norm{\eta_i}^{2s_i} $. Therefore, 
\begin{equation*}
    I_{d+2S}(\mu) \leq C_{A,\mathbf s} I_{A,\mathbf s}(\mu).
\end{equation*}
\begin{corollary}
\label{Energy}
    For each $\mu$, finite compactly supported Borel measure on $\Sone \times\R^d$,
    if $\mu \in L^2\bigl(\mathbb S^1,\dot H_A^{\mathbf s}(\R^d)\bigr)$ then $$\int_{\Sone} I_{d+2S}(\mu_x) dx < +\infty.$$
\end{corollary}

\subsection{Proof of Theorem \ref{thm:3}}

\begin{proof}
\emph{Part \textup{(1)}.}
Assume that
$-\frac{d_i}{2}<s_i<0$ and $S:=\sum_{i=1}^t s_i$.
Applying Corollary~\ref{Energy} to $\mu_\phi$, we obtain
\begin{equation*}
\int_{\Sone} I_{d+2S}(\mu_{x})dx<\infty.
\end{equation*}
Hence $
I_{d+2S}(\mu_{x})<\infty
$
for Lebesgue almost every $x\in\Sone$. The energy criterion for
Hausdorff dimension therefore yields
\begin{equation*}
\DimH (\mu_{\phi,x})\geq d+2S
\end{equation*}
for almost every $x\in\Sone$.
\\ \\

\emph{Part \textup{(2)}.} 
Assume that $s_i\geq 0$ for every $i=1,\ldots,t$. By
Lemma~\ref{lemma:relation_Sobolev_norms} we have
\(
\mu\in L^2\bigl(\mathbb S^1,H_A^{\bs}(\mathbb R^d)\bigr).
\)
Therefore,
\[
\int_{\mathbb S^1}\int_{\mathbb R^d}
\left|\mathcal F_y\mu_\phi(x,\eta)\right|^2
\,d\eta\,dx \leq
\int_{\mathbb S^1}\int_{\mathbb R^d}
\left|\mathcal F_y\mu_\phi(x,\eta)\right|^2
\prod_{i=1}^t\langle\eta_i\rangle^{2s_i}
\,d\eta\,dx
<\infty.
\]

Thus, for almost every $x\in\Sone$,
$\mathcal{F}_y{\mu_{x}}\in L^2(\R^d).$
By Plancherel's theorem, there exists $h_x\in L^2(\R^d)$ such that
\[
    d\mu_{x}(y)=h_x(y)\,dy.
\]
Moreover, $\int_{\Sone}\|h_x\|_{L^2(\R^d)}^2\,dx<\infty.$
Therefore, the function $h(x,y):=h_x(y)$ belongs to
$L^2(\Sone\times\R^d)$ and
\[
    d\mu(x,y)=h(x,y)\,dx\,dy.
\]
Hence $\mu$ is absolutely continuous with respect to the Lebesgue measure.\\

\emph{Part \textup{(3)}.}
By Lemma~\ref{lemma:relation_Sobolev_norms}, since $\mu$ is supported in a fixed compact subset of $\mathbb R^d$, we have
$
\mu\in L^2\bigl(\mathbb S^1,H_A^{\mathbf s}(\mathbb R^d)\bigr)
$ for every $\mathbf s\in S.
$
We choose $\mathbf{s}^{(i)}$ as any one of the vectors in $\mathcal{S}$ such that $s^{(i)}_i=\max_{\mathbf{s} \in \mathcal{S}}s_i$. Then
\begin{align*}
\|\mu\|_{L^2(\mathbb S^1,H^s(\mathbb R^d))}^2
&=
\int_{\mathbb S^1}\int_{\mathbb R^d}
|\mathcal F_y\mu(x,\eta)|^2
\langle\eta\rangle^{2s}\,d\eta\,dx
\lesssim
\int_{\mathbb S^1}\int_{\mathbb R^d}
|\mathcal F_y\mu(x,\eta)|^2
\sum_{i=1}^t
\langle\eta_i\rangle^{2s}\,d\eta\,dx\\&
\leq \sum_{i=1}^t \int_{\mathbb S^1}\int_{\mathbb R^d} |\mathcal F_y\mu(x,\eta)|^2
\prod_{j=1}^t \langle\eta_j\rangle^{2s^{(i)}_j} d\eta dx =
\sum_{i=1}^t
\|\mu\|_{L^2(\mathbb S^1,H_A^{\textbf{s}^{(i)}}(\R^d))}^2
<\infty. 
\end{align*}
This means that $\mu \in L^2\bigl(\mathbb S^1,H^{s}(\R^d)\bigr)$.
\end{proof}

\section{Geometric Anisotropic Spaces}\label{section: anisotropic spaces}
In this section we prove some preliminary results that are necessary for the proof of Theorem \ref{thm:1} in Section \ref{section: proof of theorem 1}. More precisely, we study the Perron-Frobenius operator on a suitable \emph{geometric} anisotropic Banach space in the spirit of \cite{liverani-gouezel-06,gouezel-liverani-08} extending the techniques of \cite{avila-gouezel-tsujii-06} to the higher-dimensional setting.

\subsection{The transfer operator}

Let $P$ be the Perron--Frobenius operator associated to $T$, that is
\begin{align*}
    Pu(\mathbf x)
    &=
    \frac{1}{\ell|\det(A)|}
    \sum_{\mathbf y\in T^{-1}(\mathbf x)}
    u(\mathbf y) \\
    &=
    \frac{1}{\ell|\det(A)|}
    \sum_{a\in\A}
    u\bigg(
        \tau_a^{-1}(x),
        A^{-1}(y-\phi(\tau_a^{-1}(x)))
    \bigg),
\end{align*}

More generally, for a word $\ba\in\A^q$, the inverse branch of $T^q$
corresponding to $\ba$ is given by
\[
    (x,y)
    \longmapsto
    \bigg(
        \tau_{\ba}^{-q}(x),
        A^{-q}(y-S(x,\ba))
    \bigg).
\]
The Jacobian of $T^q$ is $\ell^q|\det(A^q)|$ and hence the branch operator
associated to $\ba$ is
\[
    P_{\ba}^q u(x,y)
    =
    \frac{1}{|\det(A^q)|\ell^q}
    u\bigg(
        \tau_{\ba}^{-q}(x),
        A^{-q}(y-S(x,\ba))
    \bigg).
\]
Thus
\[
    P^q u
    =
    \sum_{\ba\in\A^q}P_{\ba}^q u.
\]

\subsection{The anisotropic geometric norms}
Following \cite{liverani-gouezel-06}, we introduce a family of admissible stable manifolds which generalize to $\Sone\times\R^d$ the admissible stable curves considered in \cite{avila-gouezel-tsujii-06}. Let $W\colon\cD(W)\to \Sone\times\R^d$ be a $C^r$ manifold whose domain of definition $\cD(W)$ is a compact $d$-dimensional ball in $\R^d$. Then, for any $n\in\N$ there exist $\ell^n$ manifolds $W_i\colon\cD(W)\to \Sone\times\R^d,$ $j=1,\dots, l^n$, such that $T^n\circ W_j=W.$ Let $\cW$ be the family of $C^r$ manifolds $W\colon\cD(W)\to\R^d$ such that $W(t)=(\pi\circ W(t), t),$ where $\pi\colon \Sone\times \R^d\to\Sone$ is the projection on the first component. Moreover, there exist constants $c_i,$ $i=0,\dots, r,$ such that $$\norm{d^i(\pi\circ W)(t)}\leq c_i,$$
for all $i=0,\dots, r$ and all $t\in\cD(W)$. If we take any $W_j$ such that $T^n\circ W_j=W,$ since $T$ is uniformly hyperbolic, we can reparameterize the manifold $\tilde{W}_j\circ g=W_j$ so that $\tilde{W}_j\colon\cD(\tilde{W}_j)\to \R^d$ and $$\norm{d^i(g^{-1})(t)}\leq c|\det(A)|^{n},$$
for any $0\leq i\leq r$ and all $t\in\cD(\tilde{W}_j).$
In particular, we assume $c_1=\alpha,$ so that all tangent vectors to the manifolds $W\in\cW$ belong to the $DT^{-N}$-invariant cone $$C=\left\{(u,v)\in\R \times \R^d: |u|\leq \alpha^{-1}\norm{v}\right\}.$$

We can now define the anisotropic norm. Let us fix $\rho\in\N$ such that $\rho\leq r-1.$ For a function $u\in C^r(D),$ we define 
$$\norm{u}_\rho
    :=
    \max_{\alpha+\underline{\beta}\leq \rho}
    \sup_{W\in\cW}
    \,\ \sup_{\psi\in C^{\alpha+\underline{\beta}}(W)}
    \left|
        \int_{\cD(W)}
        \psi(s)\,
        \partial_x^\alpha\partial_{\underline{y}}^{\underline{\beta}} u(W(s))\,ds
    \right|,$$
where the last supremum is taken over all test functions $\psi$ such that
\[
    \supp  \psi\subset \operatorname{Int}(D(\gamma)),
    \qquad
    \|\psi\|_{C^{\alpha+\underline{\beta}}}\leq 1.
\]
Notice that $\underline{\beta}=(\beta_1,\dots,\beta_d)$ is a vector which represents the number of \emph{stable} derivatives along the directions $(y_1,\dots,y_d).$ With a slight abuse of notation we write $a+\underline{\beta}\leq \rho$, meaning that $\alpha+\sum_{i=1}^{d}\beta_i\leq \rho.$ 
The norm $\norm{\cdot}_\rho$ satisfies the following Lasota-Yorke type inequality. 
\begin{proposition}
\label{prop:LY_anisotropic}
    For any integer $0\leq\rho\leq r-1$  there exists a constant $C>0$ for which
    \begin{equation}
        \label{eq:LY1}
        \norm{P^nu}_\rho\leq C\norm{u}_\rho,
    \end{equation}
    for all $u\in C^r(D)$ and all $n\in\N.$
    In addition, if $1\leq\rho\leq r-1,$ 
    \begin{equation}
    \label{eq:LY2}
        \norm{P^nu}_\rho\leq C\ell^{-\rho n}\norm{u}_\rho+C_n\norm{u}_{\rho-1},
    \end{equation}
    where $C_n$ is a constant that may depend on $n.$
\end{proposition}
The proof of Proposition \ref{prop:LY_anisotropic} follows from that of \cite[Lemma 5]{avila-gouezel-tsujii-06}, with few adjustments due to the increase of dimension.

\begin{remark}
    The norm $\norm{\cdot}_\rho$ is a particular example of anisotropic \emph{geometric} norm, firstly introduced in \cite{blank-keller-liverani-02} and then refined in \cite{liverani-gouezel-06,gouezel-liverani-08}. In fact, if we define $\cB_\rho$ as the completion of $C^r(D)$ with respect to $\norm{\cdot}_\rho,$ then its elements are distributions that behave as smooth functions in the unstable direction, while they are singular in the stable direction. By definition, every smooth function $u\in C^r(D)$ belongs to $\cB_\rho,$ but the anisotropic Banach space is strictly larger than $C^r(D).$ For instance, let $U$ be an horizontal segment in the cylinder $\Sone\times\R^d$ and consider a compactly supported $C^r(D)$ density $\phi$ on $U.$ Then it is not difficult to show that the distribution 
    $$\phi(u):=\int_U u \phi$$ 
    belongs to $\cB_\rho$ (see \cite[Proposition 4.4]{liverani-gouezel-06} for the proof). 

    In addition, we notice that our anisotropic norm depends only on a parameter $\rho,$ while, in general, there are two indexes $p$ and $q$ that represent the regularity in the unstable direction and the negative regularity in the stable direction, respectively. In our setting, $p=\rho$ and $q=0.$ $\cB_\rho$ is not the right space if one wants to prove that $P$ is a quasi-compact operator, because $\cB_\rho$ is not compactly included in $\cB_{\rho-1}$ and it is not possible to apply Hennion's Theorem \cite{hennion-93}. On the other hand, we do not aim to find the SRB measure in any anisotropic Banach space $\cB_\rho,$ but our goal is to prove that it belongs to the anisotropic (homogeneous) Sobolev space $L^2(\Sone,\dot{H}_A^\bs(\R^d)).$
\end{remark}

\subsection{Applications to Fourier Decay}

In this section we provide a connection between geometrical information and Fourier analysis of our system. We start by showing that elements of the geometric anisotropic spaces, introduced in the previous section, have good Fourier decay in all direction except a small cone.\newline

Let $C^*$ be the cone in $\mathbb R \times \R^d$ defined by
\[
    C^*
    =
    \left\{
        (\xi,\eta)\in\R \times \R^d
        : \norm{\eta}\leq \alpha^{-1}|\xi|
    \right\},
\]
so that
\[
    (DT^N_x)^*(C^*)\subset C^*
    \qquad
    \text{for } x\in \Sone \times\R^d.
\]

\begin{lemma}
\label{Fdecay}
    Let $\rho\leq r-1$, $q\geq N$ and $p\geq0$ be integers, $\ba$ and $\bc$ elements of $\A^q$ and $\A^p$, respectively, and $\chi : \Sone \times \R^d \rightarrow \R$ a $C^\infty$ function supported on $\Pcal_*(\bc\ba) \times \R^d$. Take $(\xi,\eta) \in \mathbb{Z}\times \R^d\setminus\{(0,0)\}$ such that, for any $x \in \Pcal_*(\bc\ba) \times \R^d$, $(DT^q_x)^*(\xi,\eta) \in C^*$. Then, for any $\varphi\in C^r(\Sone\times\R^d)$,
    
\begin{equation*}
     (\xi^2 + \norm{\eta}^2)^{\rho/2} \mathcal{F} (P^q(\chi\varphi))(\xi,\eta) \leq C(q,\chi) \|\varphi\|_\rho,
\end{equation*}
where $C(q,\chi)$ may depend on $q$ and $\chi$.
\end{lemma}
\begin{proof}
Let $(\xi,\eta)\in\Z\times\R^d\setminus\{(0,0)\}$ be a covector satisfying the assumption of the lemma. By applying the properties of the Fourier transform, the left hand side is bounded, up to a multiplicative constant $C,$ by 
\begin{equation}
\label{eq:absolute_value}
    \left|\cF(\partial^\rho P^q(\chi\varphi))(\xi,\eta) \right|,
\end{equation}
where $\partial$ denotes the partial derivative dual to the dominating term of the sum $(\xi^2+\norm{(\eta_1,\eta_2,\dots,\eta_d)}^2).$ In particular, if $|\xi|\geq\max\{|\eta_i|,\ {i=1,\dots,d}\},$ then $\partial=\partial_x;$ otherwise, $\partial=\partial_{y_j},$ where $|\eta_j|\geq\max\{|\xi|,|\eta_j|,\ j=1,\dots d\}.$ The compact domain $D\cap \Pcal_*(\bc)\times\R^d$ can be covered by a 1-parameter family of balls $\cB=\{B_t\}_{t\in[0,1]}$, orthogonal to $(\xi,\eta)$ and whose radii are uniformly bounded from above.  By computing the Fourier transform with this parametrization, \eqref{eq:absolute_value} is bounded by $$ C\sup_{t\in[0,1]}\int_{\cD(B_t)}\partial^\rho P^q(\chi\varphi)(B_t(s))ds.$$
For any $B_t,$ there exists a unique preimage $\bar B_t\in\cW$ contained in $\Pcal_*(\bc\ba)\times\R^d.$ Hence, after a change of coordinates $$\int_{\cD(B_t)}\partial^\rho P^q(\chi\phi)(B_t(s))ds=\int_{\cD(\bar B_t)}\partial^\rho\varphi(B_t(s))ds.$$ Since, for any tangent vector $(u,v)$ to $B_t,$ $0=\langle(u,v),(\xi,\eta)\rangle=\langle DT^{-q}(u,v),(DT^q)^*(\xi,\eta)\rangle,$ and $(DT^q)^*(\xi,\eta)\in C^*,$ we conclude that $DT^{-q}(u,v)\in C.$ Therefore, up to a reparameterization, $\bar B_t$ belongs to the set of admissible stable manifolds $\cW.$ Accordingly, the last integral is bounded by a constant multiple of $\norm{\varphi}_\rho.$
  
\end{proof}

The next lemma is the link between the derivative dispersion condition (or any notion of transversality) and Fourier decay (and therefore Sobolev regularity) of the density of the invariant measure.

\begin{lemma}
    Assume that $(\xi,\eta) \in \mathbb{Z}\times\R^d\setminus\{(0,0)\}$, $\ba \in G^q_\bc(\xi,\eta)$, $x\in\mathcal{P}_*(\bc)$ and $y\in\R^d$. Then $(DT^q_{(\ba(x),y)})^*(\xi,\eta) \in C^*$.
\end{lemma}
\begin{proof}
Let $\ba\in G_\bc^q(\xi,\eta)$, then, by definition, for all $x\in\Pcal_*(\bc),$  
\begin{equation}
    \label{eq:transversality}
 \left|\xi+\eta\cdot S'_\bc(x,\ba)\right|
 \geq 2\alpha\ell^{-q} \norm{(A^q)^T\eta}.
\end{equation}
 A straightforward computation gives 
 \begin{equation*}
     (DT^q_{(\ba(x),y)}))^*=\left(\begin{matrix}
     l^q&l^q(S'_\bc(x,\ba))^T\\\underline{0}&(A^q)^T
 \end{matrix}\right),
 \end{equation*}
writing $$\left(\begin{array}{c}
      \xi'\\\eta'    
 \end{array}\right)=\left(\begin{matrix}
     l^q&l^q(S'_\bc(x,\ba))^T\\\underline{0}&(A^q)^T
 \end{matrix}\right)\left(\begin{array}{c}
      \xi\\\eta    
 \end{array}\right),$$
we obtain that $\norm{\eta'}\leq\alpha^{-1}|\xi'|,$ i.e., $(\xi',\eta')\in C^*$, if $$\norm{(A^q)^T\eta}\leq\alpha^{-1}l^q|\xi+S'_\bc(x,\ba)\cdot\eta|,$$ which follows from the condition \eqref{eq:transversality}.

\end{proof}

\section{Proof of Theorem \ref{thm:1} and Proposition \ref{prop:Proposition_regularity}}\label{section: proof of theorem 1}

The core of the proof of Theorem \ref{thm:1} is another Lasota-Yorke inequality, proved in the next subsection, that relates the anisotropic Sobolev norm to the geometric anisotropic norm. The theorem then follows by a simple functional analysis argument. 

\subsection{A Lasota-Yorke Inequality} 
We denote by $ s_i^+:=\max\{s_i,0\},$ and $s_i^-:=\max\{-s_i,0\}. 
$ \begin{lemma} 
\label{lemma:LY_1}
For every $u\in C^r(D)$, \[ \bigl\|\Pq_\ba(\chi_{\bc\ba}u)\bigr\|_{L^2(\Sone,\dot H_A^\bs(\R^d))}^2 \leq \frac{1}{\ell^q|\det A|^q} \prod_{i=1}^t \left( \|B_i^{-q}\|^{2s_i^+} \|B_i^q\|^{2s_i^-} \right) \|\chi_{\bc\ba}u\|_{L^2(\Sone,\dot H_A^\bs(\R^d))}^2. \] 
\end{lemma}

\begin{proof} Set $v:=\chi_{\bc\ba}u$. Since 
\begin{equation}
\label{eq:transfer_operator_P_a}
    P_{\ba}^qv(x,y) = \frac{1}{|\det A|^q\ell^q} v\left( \tau_{\ba}^{-q}(x), A^{-q}\bigl(y-S_\ba(x)\bigr) \right), 
\end{equation} 
taking the Fourier transform in $y$ gives 
\begin{equation}
\label{eq:Fourier_transform_transfer_operator}
 \mathcal F_y(\Pq_\ba v)(x,\eta) = \ell^{-q} e^{-i\langle\eta,S_\ba(x)\rangle} \mathcal F_yv \bigl( \tau_{\ba}^{-q}(x), (A^q)^T\eta \bigr).   
\end{equation}
Therefore, 
\begin{align*} 
\|\Pq_\ba v\|_{L^2(\Sone,\dot H_A^\bs(\R^d))}^2 &= \int_{\Sone} \int_{\R^d} \left| \mathcal F_y(\Pq_\ba v)(x,\eta) \right|^2 \prod_{i=1}^t\norm{\eta_i}^{2s_i} \,d\eta\,dx \\ &= \ell^{-2q} \int_{\Sone} \int_{\R^d} \left| \mathcal F_yv \bigl( \tau_{\ba}^{-q}(x), ((A^q)^T\eta \bigr) \right|^2 \prod_{i=1}^t\norm{\eta_i}^{2s_i} \,d\eta\,dx \\ &= \ell^{-q} \int_{\Sone} \int_{\R^d} \left| \mathcal F_yv \bigl( x', ((A^q)^T\eta \bigr) \right|^2 \prod_{i=1}^t\norm{\eta_i}^{2s_i} \,d\eta\,dx' \\ &= \frac{1}{\ell^q|\det A|^q} \int_{\Sone} \int_{\R^d} \left| \mathcal F_yv(x',\zeta) \right|^2 \prod_{i=1}^t \norm{B_i^{-q}\zeta_i}^{2s_i} \,d\zeta\,dx'. \end{align*} 

If $s_i\geq0$, then $ \norm{B_i^{-q}\zeta_i} \leq \|B_i^{-q}\|\norm{\zeta_i}, $ and hence $\norm{B_i^{-q}\zeta_i}^{2s_i} \leq \|B_i^{-q}\|^{2s_i} \norm{\zeta_i}^{2s_i}. $ If $s_i<0$, then $ \norm{B_i^{-q}\zeta_i} \geq \|B_i^q\|^{-1} \norm{\zeta_i}, $ and hence $ \norm{B_i^{-q}\zeta_i}^{2s_i} \leq \|B_i^q\|^{-2s_i} \norm{\zeta_i}^{2s_i}. $ 
Therefore, 
\begin{align*} \|\Pq_\ba v\|_{L^2(\Sone,\dot H_A^\bs(\R^d))}^2 &\leq \frac{1}{\ell^q|\det A|^q} \prod_{i=1}^t \left( \|B_i^{-q}\|^{2s_i^+} \|B_i^q\|^{2s_i^-} \right) \int_{\Sone} \int_{\R^d} \left| \mathcal F_yv(x',\zeta) \right|^2 \prod_{i=1}^t \norm{\zeta_i}^{2s_i} \,d\zeta\,dx' \\ &= \frac{1}{\ell^q|\det A|^q} \prod_{i=1}^t \left( \|B_i^{-q}\|^{2s_i^+} \|B_i^q\|^{2s_i^-} \right) \|v\|_{L^2(\Sone,\dot H_A^\bs(\R^d))}^2 \\ &= \frac{1}{\ell^q|\det A|^q} \prod_{i=1}^t \left( \|B_i^{-q}\|^{2s_i^+} \|B_i^q\|^{2s_i^-} \right) \|\chi_{\bc\ba}u\|_{L^2(\Sone,\dot H_A^\bs(\R^d))}^2. 
\end{align*}  
\end{proof}

\begin{lemma}
\label{lemma: estimate_norm}
For $s_i>-d_i/2$ and $ \rho>S+\frac{d+1}{2},$ we have \[ \sum_{\xi\in\mathbb Z} \int_{\mathbb R^d} (1+\xi^2+\norm{\eta}^2)^{-\rho} \prod_{i=1}^t\norm{\eta_i}^{2s_i}\,d\eta <\infty. \] \end{lemma}

\begin{proof} 
Since $ \rho>S+\frac{d+1}{2}>\frac12, $ we have \[ \sum_{\xi\in\mathbb Z} (1+\xi^2+r^2)^{-\rho} \lesssim (1+r^2)^{\frac12-\rho}. \] Moreover, the function $ \prod_{i=1}^t\norm{\eta_i}^{2s_i}$ is homogeneous of degree $2S$ in $\eta$. Since $ s_i>-\frac{d_i}{2}, $ its restriction to the unit sphere is integrable. Hence, using polar coordinates in $\mathbb R^d$, 
\begin{align*} & \sum_{\xi\in\mathbb Z} \int_{\mathbb R^d} (1+\xi^2+\norm{\eta}^2)^{-\rho}\prod_{i=1}^t\norm{\eta_i}^{2s_i} \,d\eta \\ &\qquad\lesssim \int_0^\infty r^{d-1+2S} (1+r^2)^{\frac12-\rho} \,dr < \infty. 
\end{align*}
\end{proof} 

\begin{proposition}
\label{prop:LY}
For $ s_i>-d_i/2,$ and an integer $ \rho>S+\frac{d+1}{2},$ there exists a constant $C(q)$ such that, for all $u\in C^r(D)$ with $r\geq\rho+1$, \[ \|P^qu\|_{L^2(\Sone,\dot H_A^\bs(\R^d))}^2 \leq \frac{36M(q)} {\ell^q|\det A|^q} \prod_{i=1}^t \left( \|B_i^{-q}\|^{2s_i^+} \|B_i^q\|^{2s_i^-} \right) \|u\|_{L^2(\Sone,\dot H_A^\bs(\R^d))}^2 + C(q)\|u\|_{\rho}^2. \] 
\end{proposition} 

\begin{proof} Fix $q,p\geq N$. Let $\{\chi_\bc:\Sone\to[0,1],\ \bc\in\mathcal A^p\}$ be a $C^\infty$ partition of unity subordinate to the covering $\{\operatorname{Int}\Pcal^*(\bc),\ \bc\in\mathcal A^p\}$. For $\ba\in\mathcal A^q$, define $\chi_{\bc\ba}$ by \[ \chi_{\bc\ba} \bigl( \tau_{\bc\ba}^{-q}(x) \bigr) = \chi_\bc(x), \qquad x\in\Pcal^*(\bc), \] and extend it by zero elsewhere. 

Then $\{\chi_{\bc\ba},\ (\ba,\bc)\in \mathcal A^q\times\mathcal A^p\}$ is again a $C^\infty$ partition of unity. By construction, \[ \chi_\bc P^qu = \sum_{\ba\in\mathcal A^q} P_\ba^q(\chi_{\bc\ba}u). \] 
To simplify the notation we set $ u_{\ba,\bc} := \Pq_\ba(\chi_{\bc\ba}u). $ By the properties of the partition of the unity, (in particular every point $x \in \Sone$ belongs to the support of at most three elements of the partition), and the definition of the norm we get 

\begin{equation}
\label{P1}\|u\|_{L^2(\Sone,\dot H_A^\bs(\R^d))}^2 \leq 3 \sum_{\bc\in\mathcal A^p} \|\chi_\bc u\|_{L^2(\Sone,\dot H_A^\bs(\R^d))}^2 \end{equation} and \begin{equation}
\label{P2}\sum_{\ba\in\mathcal A^q} \sum_{\bc\in\mathcal A^p} \|\chi_{\bc\ba}u\|_{L^2(\Sone,\dot H_A^\bs(\R^d))}^2 \leq \|u\|_{L^2(\Sone,\dot H_A^\bs(\R^d))}^2. \end{equation} 
Hence,
\begin{align*} \|P^qu\|_{L^2(\Sone,\dot H_A^\bs(\R^d))}^2 &\leq 3 \sum_{\bc\in\mathcal A^p} \left\| \sum_{\ba\in\mathcal A^q} u_{\ba,\bc} \right\|_{L^2(\Sone,\dot H_A^\bs(\R^d))}^2 = 3 \sum_{\bc\in\mathcal A^p} \sum_{\xi\in\Z} \int_{\R^d} \left| \sum_{\ba\in\A^q} \widehat u_{\ba,\bc}(\xi,\eta) \right|^2 \prod_{i=1}^t|\eta_i|^{2s_i}\,d\eta. 
\end{align*} 

For every $\bc\in\mathcal A^p$ and $(\xi,\eta)\in\mathbb Z\times\mathbb R^d$, we write \[ \mathcal A^q = B_\bc(\xi,\eta) \cup G_\bc(\xi,\eta), \] and we obtain 

\begin{align*} 
\|P^qu\|_{L^2(\Sone,\dot H_A^\bs(\R^d))}^2 &\leq 3 \sum_{\bc\in\mathcal A^p} \sum_{\xi\in\Z} \int_{\R^d} \left| \sum_{\ba\in B_\bc(\xi,\eta)} \widehat u_{\ba,\bc}(\xi,\eta) + \sum_{\ba\in G_\bc(\xi,\eta)} \widehat u_{\ba,\bc}(\xi,\eta) \right|^2 \prod_{i=1}^t|\eta_i|^{2s_i}\,d\eta \\ &\leq 6 \sum_{\bc\in\mathcal A^p} \sum_{\xi\in\mathbb Z} \int_{\mathbb R^d} \#B_\bc(\xi,\eta) \sum_{\ba\in B_\bc(\xi,\eta)} |\widehat u_{\ba,\bc}(\xi,\eta)|^2 \prod_{i=1}^t|\eta_i|^{2s_i}\,d\eta \\ &\quad+ 6 \sum_{\bc\in\mathcal A^p} \sum_{\xi\in\mathbb Z} \int_{\mathbb R^d} \left( \sum_{\ba\in G_\bc(\xi,\eta)} |\widehat u_{\ba,\bc}(\xi,\eta)| \right)^2 \prod_{i=1}^t|\eta_i|^{2s_i} \,d\eta. 
\end{align*} 

We recall that \[ \#B_\bc(\xi,\eta) \leq M(p,q), \] and using Lemma~\ref{Fdecay} we get 
\begin{align*} \|P^qu\|^2 &\leq 6M(p,q) \sum_{\bc\in\mathcal A^p} \sum_{\xi\in\mathbb Z} \int_{\mathbb R^d} \sum_{\ba\in\mathcal A^q} |\widehat u_{\ba,\bc}(\xi,\eta)|^2 \prod_{i=1}^t|\eta_i|^{2s_i}\,d\eta \\ &\quad+ C(q,p)\|u\|_{\rho}^{2} \sum_{\xi\in\mathbb Z} \int_{\mathbb R^d} (1+\xi^2+|\eta|^2)^{-\rho}\prod_{i=1}^t|\eta_i|^{2s_i} \,d\eta \\ &= 6M(p,q) \sum_{\ba\in\mathcal A^q} \sum_{\bc\in\mathcal A^p} \|u_{\ba,\bc}\|^2 \\ &\quad+ C(q,p)\|u\|_{\rho}^{2} \sum_{\xi\in\mathbb Z} \int_{\mathbb R^d} (1+\xi^2+|\eta|^2)^{-\rho}\prod_{i=1}^t|\eta_i|^{2s_i} \,d\eta. 
\end{align*} 
By Lemma \ref{lemma:LY_1}, \eqref{P1} and \eqref{P2}
\begin{align*} \|P^qu\|_{L^2(\Sone,\dot H_A^\bs(\R^d))}^2 &\leq 6M(p,q) \sum_{\ba\in\mathcal A^q} \sum_{\bc\in\mathcal A^p} \|u_{\ba,\bc}\|_{L^2(\Sone,\dot H_A^\bs(\R^d))}^2 + C(q,p)\|u\|_{\rho}^{2} \\ &\leq \frac{6M(p,q)} {\ell^q|\det A|^q} \prod_{i=1}^t \left( \|B_i^{-q}\|^{2s_i^+} \|B_i^q\|^{2s_i^-} \right) \sum_{\ba\in\mathcal A^q} \sum_{\bc\in\mathcal A^p} \|\chi_{\bc\ba}u\|_{L^2(\Sone,\dot H_A^\bs(\R^d))}^2 + C(q,p)\|u\|_{\rho}^{\,2} \\ &\leq \frac{36M(p,q)} {\ell^q|\det A|^q} \prod_{i=1}^t \left( \|B_i^{-q}\|^{2s_i^+} \|B_i^q\|^{2s_i^-} \right) \|u\|_{L^2(\Sone,\dot H_A^\bs(\R^d))}^2 + C(q,p)\|u\|_{\rho}^{\,2}. 
\end{align*} 
Finally, for a fixed value of $q$, choose $p=p(q)$ sufficiently large so that \[ M(p,q)=M(q). \]  We obtain \[ \|P^qu\|_{L^2(\Sone,\dot H_A^\bs(\R^d))}^2 \leq \frac{36M(q)} {\ell^q|\det A|^q} \prod_{i=1}^t \left( \|B_i^{-q}\|^{2s_i^+} \|B_i^q\|^{2s_i^-} \right) \|u\|_{L^2(\Sone,\dot H_A^\bs(\R^d))}^2 + C(q)\|u\|_{\rho}^{\,2}, \] that is the wanted inequality. 

\end{proof}

Let $m_i$ be the maximal size of the Jordan block associated with the eigenvalue $\lambda_i$. If $\lambda_i$ is non real, the Jordan blocks are understood in the complexification of $B_i=A^T_{\mid E_i}$. 

\begin{lemma}
\label{lemma:upper_bound}
There exists a constant $K_{A,\bs}>0$ such that, for every $q\geq1$, \[ \prod_{i=1}^t |\lambda_i|^{-2s_iq} \leq \prod_{i=1}^t\left( \|B_i^{-q}\|^{2s_i^+} \|B_i^q\|^{2s_i^-} \right) \leq K_{A,\bs}q^k \prod_{i=1}^t |\lambda_i|^{-2s_iq}, \] where $ k = 2(d-1)\sum_{i=1}^t |s_i|. $
\end{lemma}

\begin{proof}
We consider the complexification of the matrix $A$.
Each Jordan block $J_i$ has the
form
$J_i = \lambda_i I + N,$
and
$N^{m_i}=0.
$
Hence
\begin{align*}
J_i^q
&=
\sum_{r=0}^{m_i-1}
\binom{q}{r}\lambda_i^{q-r}N^r, \quad \quad \text{and} \quad \quad
J_i^{-q}=
\lambda_i^{-q}
\sum_{r=0}^{m_i-1}
(-1)^r
\binom{q+r-1}{r}
\lambda_i^{-r}N^r .
\end{align*}
For every fixed $0\leq r\leq m_i-1$, we have
$
\binom{q}{r}\leq q^r$
and
$\binom{q+r-1}{r}\leq q^r$.
Since $r\leq m_i-1$, it follows that
$
\|J_i^q\|
\leq
C_iq^{m_i-1}|\lambda_i|^q$ and $
\|J_i^{-q}\|
\leq
C_iq^{m_i-1}|\lambda_i|^{-q}
$,  for some $C_i\geq 1.$
Conjugating by $P_i$ and enlarging $C_i$, we obtain
\begin{equation*}
\|B_i^q\|
\leq
C_iq^{m_i-1}|\lambda_i|^q,
\qquad
\|B_i^{-q}\|
\leq
C_iq^{m_i-1}|\lambda_i|^{-q}.
\end{equation*}
Moreover, since the operator norm dominates the spectral radius,
$
\|B_i^q\|\geq |\lambda_i|^q$ and
$\|B_i^{-q}\|\geq |\lambda_i|^{-q}.$
Using
$s_i=s_i^+-s_i^-$ and $s_i^++s_i^-=|s_i|$, we obtain
\begin{equation*}
|\lambda_i|^{-2s_iq}
\leq
\|B_i^{-q}\|^{2s_i^+}\|B_i^q\|^{2s_i^-}
\leq
C_i^{2|s_i|}
q^{2|s_i|(m_i-1)}
|\lambda_i|^{-2s_iq}.
\end{equation*}
We multiply over $i$ and we get
\begin{equation*}
\prod_{i=1}^t |\lambda_i|^{-2s_iq}
\leq
\prod_{i=1}^t
\left(
\|B_i^{-q}\|^{2s_i^+}
\|B_i^q\|^{2s_i^-}
\right)
\leq
K_{A,\bs} q^k
\prod_{i=1}^t |\lambda_i|^{-2s_iq},
\end{equation*}
where
$
K_{A,\bs}
:=
\prod_{i=1}^t C_i^{2|s_i|}
$ and $
k
=
2\sum_{i=1}^t |s_i|(m_i-1).
$
Finally, since $m_i\leq d_i\leq d$, we replace $k$ by the larger exponent
$
k:=2(d-1)\sum_{i=1}^t |s_i|.
$
\end{proof}

As a consequence of the previous proposition and lemma, we obtain the following Lasota-Yorke inequality. 
\begin{corollary}
\label{corollary:LY_q}
There exists a constant $B_\bs>0$, depending only on $A$ and $\bs$, such that 
\[ \|P^qu\|_{L^2(\Sone,\dot H_A^\bs(\R^d))}^2 \leq \frac{ B_\bs q^k M(q) } { \left( \ell \prod_{i=1}^t |\lambda_i|^{d_i+2s_i} \right)^q } \|u\|_{L^2(\Sone,\dot H_A^\bs(\R^d))}^2 + C(q)\|u\|_{\rho}^{\,2}. \] 
\end{corollary}
\begin{proof}
    The statement is a consequence of Proposition \ref{prop:LY} and Lemma \ref{lemma:upper_bound} once we set
    \begin{equation}
    \label{eq:constant}
            B_\bs=36 K_{A,\bs}=36\prod_{i=1}^tC_i^{2|s_i|}.
    \end{equation}
\end{proof}

\subsection{Conclusion of the proof of Theorem \ref{thm:1}}
\begin{corollary}
\label{corollary:bounded_norm}
   If $B_\bs q^k M(q) < \left( \ell \prod_{i=1}^t |\lambda_i|^{d_i+2s_i} \right)^q $ for some $q\geq N$  we have that
    \begin{equation*}
        \sup_{n\geq0} \|P^n u\|_{L^2(\Sone,\dot H_A^\bs(\R^d))} < \infty,
    \end{equation*}
    for any $u\in C^r(D)$ and any $n\in\N$.
\end{corollary}
\begin{proof}
As a consequence of Corollary \ref{corollary:LY_q}, there exists $\gamma\in(0,1)$ such that $$\norm{P^qu}_{L^2(\Sone,\dot H_A^\bs(\R^d))}\leq\gamma\norm{u}_{L^2(\Sone,\dot H_A^\bs(\R^d))}+C(q)\norm{u}_{\rho}$$
Let $n\in\N$ be a positive integer, we can write $n=aq+b,$ with $0\leq b\leq q-1.$ Then \begin{equation*}
    \begin{split}
        \norm{P^nu}_{L^2(\Sone,\dot H_A^\bs(\R^d))}&=\norm{P^{aq}(P^bu)}_{L^2(\Sone,\dot H_A^\bs(\R^d))}\leq\gamma^a\norm{P^bu}_{L^2(\Sone,\dot H_A^\bs(\R^d))}+C(q)\norm{P^bu}_{\rho}\leq\\&\leq\max_{i=0,\dots,q-1}\norm{P^iu}_{L^2(\Sone,\dot H_A^\bs(\R^d))}+C(q)C\norm{u}_{\rho}<\infty,
    \end{split}
\end{equation*}
where we applied \eqref{eq:LY1} to bound the anisotropic norm. 
\end{proof}

\begin{lemma}
    If $B_\bs q^k M(q) < \left( \ell \prod_{i=1}^t |\lambda_i|^{d_i+2s_i} \right)^q $ for some $q\geq N$,
    the SRB measure belongs to the space $L^2(\Sone,\dot H^\bs_A(\R^d))$
\end{lemma}
\begin{proof} 
Consider $h_0\in C^r(D)$ with compact support such that $h_0\geq 0$ and $\int h_0 d\Leb=1.$ Then, by the classical theory of uniformly hyperbolic dynamical systems, the measure $T_*^n(h_0d\Leb)$ converges, in the weak sense, to the unique SRB measure $\mu_\phi,$
\begin{equation}
\label{eq:weak_convergence}
    \lim_{n\to\infty}\int_{\Sone\times\R^d} P^n h_0(x,y) \psi(x,y)dxdy=\mu_\phi(\psi)=\int_{\Sone\times\R^d}\psi(x,y)d\mu_\phi(x,y).
\end{equation}
Since the weak convergence of measures implies pointwise convergence of Fourier transforms, we have $$\lim_{n\to\infty}\cF(P^n h_0)(\xi,\eta)=\cF(\mu_\phi)(\xi,\eta),$$
for any $(\xi,\eta)\in\Z\times\R^d.$ In view of Corollary \ref{corollary:bounded_norm} and Fatou's Lemma, we obtain
\begin{equation*}
    \begin{split}
        &\sum_{\xi\in\Z}\int_{\R^d}|\cF(\mu_\phi)(\xi,\eta)|^2\prod_{i=1}^t\norm{\eta_i}^{2s_i}d\eta=\sum_{\xi\in\Z}\int_{\R^d}\left|\liminf_{n\to\infty}\cF(P^nh_0)(\xi,\eta)\right|^2\prod_{i=1}^t\norm{\eta_i}^{2s_i}d\eta\leq\\\leq& \liminf_{n\to\infty}\sum_{\xi\in\Z}\int_{\R^d}\left|\cF(P^n h_0)(\xi,\eta)\right|^2\prod_{i=1}^t\norm{\eta_i}^{2s_i}d\eta\leq C
    \end{split}
\end{equation*}
which implies that $\mu_\phi\in L^2(\Sone,\dot H_A^\bs(\R^d)).$
\end{proof}

\subsection{Isotropic Sobolev regularity}
The goal of this section is the proof of Proposition \ref{prop:Proposition_regularity} of the isotropic Sobolev regularity under the assumption that the density $\Psi_\phi\in L^2(\Sone, H^s(\R^d))$ for some $s>0,$ such that $|\det(A)|\rho(A^{-1})^{-2s}>1/l^3$. We need the following preliminary result.
\begin{lemma}
   \label{lemma:bound_spectral_radius_in_Sobolev} 
   Let $u\in L^2(\Sone,H^s(\R^d))$ with $s>0.$ For any $q\in\Z_+$, 
   $$\norm{P^qu}_{L^2(\Sone,H^s(\R^d))}^2\leq C \frac{q^{2s(d-1)}\rho(A^{-q})^{2s}}{l^q|\det(A^q)|}\norm{u}_{L^2(\Sone,H^s(\R^d))}^2,$$
 In particular, the spectral radius of the transfer operator $P$ satisfies
   $$\rho(P)\leq\left(\frac{\rho(A^{-1})^{2s}}{\ell |\det(A)|}\right)^{\frac{1}{2}}$$
\end{lemma}
\begin{proof}
   The strategy of the proof is similar to the one of Lemma \ref{lemma:LY_1}, with a different weight. 
   By the definition of $P$, $\eqref{eq:transfer_operator_P_a}$ and \eqref{eq:Fourier_transform_transfer_operator}, we can compute 
   \begin{equation*}
       \begin{split}
       \norm{P^qu}_{L^2(\Sone,H^s(\R^d))}^2\leq&\sum_{\ba\in\A^q}\norm{P_\ba^qu}_{L^2(\Sone,H^s(\R^d))}^2=\sum_{\ba\in\A^q}\int_{\Sone}\int_{\R^d}|\mathcal{F}_y(P_{\ba}^qu)(x,\eta)|\langle\eta\rangle^{2s}d\eta dx=\\
           =&\sum_{\ba\in\A^q}\ell^{-2q} \int_{\Sone} \int_{\R^d} \left| \mathcal F_y u \bigl( \tau_{\ba}^{-q}(x), (A^q)^T\eta \bigr) \right|^2 \langle\eta\rangle^{2s}dx \\  
           =& \frac{1}{\ell^q|\det A|^q} \int_{\Sone} \int_{\R^d} \left| \mathcal F_yu(x',\zeta) \right|^2 \langle (A^{-q})^T\zeta\rangle^{2s} \,d\zeta\,dx'.
       \end{split}
   \end{equation*}
   Since we can bound $\norm{A^{-q}}\leq Cq^{d-1}\rho(A^{-1})^q,$ for some constant $C\geq 1$ independent of $q$, we obtain $\langle(A^{-q})^T\zeta\rangle\leq Cq^{m-1}\rho(A^{-1})^q\langle\zeta\rangle.$  Thus, 
   \begin{equation*}
       \norm{P^qu}_{L^2(\Sone,H^s(\R^d))}^2\leq C \frac{q^{2s(m-1)}\rho(A^{-1})^{2sq}}{\ell^q|\det(A)|^{q}}\norm{u}_{L^2(\Sone,H^s(\R^d))}^2.
   \end{equation*}
   The bound on the spectral radius is a consequence of Gelfand's formula 
   $$\rho(P)=\lim_{q\to\infty}\norm{P^q}_{L^2(\Sone,H^s(\R^d))}^{\frac{1}{q}}\leq \lim_{q\to\infty}C^{\frac{1}{2q}}\frac{q^{\frac{s(d-1)}{q}}\rho(A^{-1})^{s}}{\sqrt{\ell|\det(A)|}}=\frac{\rho(A^{-1})^{s}}{\sqrt{\ell|\det(A)|}}.$$
\end{proof}

\begin{proof}[Proof of Proposition \ref{prop:Proposition_regularity}]
Let us firstly assume $s>1.$ By construction, $P\Psi_\phi=\Psi_\phi$ and we can compute the weak derivative $\partial_x$ of both sides of the equality: 
\begin{equation*}
\partial_x \Psi_\phi=\ell^{-1}\partial_x P\Psi_{\phi}-\ell^{-1}P\langle\nabla_y\Psi_\phi,A^{-1}\partial_x\phi\rangle,    
\end{equation*}
and, hence,
$$(I-\ell^{-1}P)\partial_x \Psi_\phi=-\ell^{-1}P\langle\nabla_y\Psi_\phi,A^{-1}\partial_x\phi\rangle.$$
By the assumption and Lemma \ref{lemma:bound_spectral_radius_in_Sobolev}, the spectral radius of $P$ is strictly smaller than $\ell$ in $L^2(\Sone,H^{s-1}(\R^d))$. Thus $(I-\ell^{-1}P)$ is invertible,  
\begin{equation}
    \label{eq:weak_derivative}
    \partial_x \Psi_\phi=-(I-\ell^{-1}P)^{-1}\ell^{-1}P\langle\nabla_y\Psi_\phi,A^{-1}\partial_x\phi\rangle
\end{equation}
and we get $\partial_x\Psi_\phi\in L^2(\Sone,H^{s-1}(\R^d)).$ If $s$ is an integer we can conclude by iterating the same strategy $s$ times. It remains to consider the case $s\in(0,1).$ \eqref{eq:weak_derivative} is still true, but $\partial_x\Psi_\phi$ is now a distribution in $L^2(\Sone, H^{s-1}(\R^d)).$ Therefore, there exists $C_1>0$ such that 
$$\sum_{\xi\in\Z}\int_{\mathbb R^d}|\mathcal F_y\Psi_\phi(\xi,\eta)|^2|\xi|^2\langle\eta\rangle^{2(s-1)}\,d\eta\,<C_1.$$
Moreover, by the assumption $\Psi_\phi\in L^2(\Sone, H^{s}(\R^d))$, hence there exists $C_2>0$ such that
$$\sum_{\xi\in\Z}\int_{\mathbb R^d}|\mathcal F_y\Psi_\phi(\xi,\eta)|^2\langle\eta\rangle^{2s}\,d\eta\,<C_2.$$ We finally apply an interpolation argument, based on H\"older inequality, to obtain 
\begin{equation*}
    \begin{split}
        &\sum_{\xi\in\Z}\int_{\R^d}|\mathcal F\Psi_\phi(\xi,\eta)|^2|\xi|^{2s}d\eta\leq\\\leq&  \left(\sum_{\xi\in\Z}\int_{\mathbb R^d}|\mathcal F_y\Psi_\phi(\xi,\eta)|^2|\xi|^2\langle\eta\rangle^{2(s-1)}\right)^{s}\left(\sum_{\xi\in\Z}\int_{\mathbb R^d}|\mathcal F_y\Psi_\phi(\xi,\eta)|^2\langle\eta\rangle^{2s}\right)^{1-s}\leq\\\leq& C_1^{s}C_2^{1-s}<\infty
    \end{split}
\end{equation*}
We conclude that $\partial_x^{s}\Psi_\phi\in L^2(\Sone\times\R^d).$
\end{proof}

\section{Genericity of the $(\pmb{D})$- Condition}\label{section: proof of genericity}
This section is devoted to the proof of Theorem \ref{genericity}: the prevalence of the derivative dispersion condition, in a measure theoretic sense. We follow the ideas of the proof of genericity of the transversality condition in \cite{tsujii-01}. We first prove that the graphs $S(x,\ba)$ are very sensitive to suitable perturbations of $\phi$, expressing this concept through a notion of Jacobian. Then we conclude with a probabilistic argument, following the exposition of \cite{avila-gouezel-tsujii-06}.

\subsection{Generic Families}
For $\phi\in C^2(\Sone,\R^d)$ and a finite collection of smooth maps
$\left\{\psi_i\right\}_{i\in\mathcal{I}} \subset C^\infty(\Sone,\R^d)$, we define the family of functions
\begin{equation*}
\phi_\bt(x)=\phi(x)+\sum_{i\in\mathcal{I}}t_i\psi_i(x),
\end{equation*}
where the parameter $\bt$ ranges in $\R^{\mathcal{I}}$.\newline
Given $A\in GL(d,\R)$, with $\rho(A)<1$, we consider the corresponding family of maps, from $\Sone\times\mathbb R^d$ to itself,
\begin{equation*}
  T_\bt(x,y)=(\ell x, A y+\phi_\bt(x)), 
\end{equation*}
together with the family of graphs
\begin{equation*}
  S(x,\ba;\bt)=\sum_{i=1}^{n}A^{i-1}\phi_\bt([\ba]_i(x)),
\end{equation*}
for any word $\ba\in\A^n$ of length $1\le n\le\infty$.\newline
For a point $x\in \Sone$ and a sequence $\sigma=(\ba_1,\ldots,\ba_k)$ of elements in $\A^\infty$, we
consider an affine map $G_{x,\sigma}:\mathbb R^{\mathcal{I}}\to\mathbb R^{dk}$ defined by
\begin{equation*}
  G_{x,\sigma}(\bt)=
  \left(
    \frac{d}{dx}S(x,\ba_i;\bt))
  \right)_{i=1,2,\ldots,k}.
\end{equation*}
For a linear map $L$ between vector spaces $V$ and $W$, of possibly different dimensions, we define the Jacobian of $L$, if it is surjective, as
\begin{equation*}
    \Jac(L) = \sup_{\dim(E)=\dim(W)} \Jac(L|_E),
\end{equation*}
where $E$ ranges among all subspaces of $V$ with the same dimension as $W$. We set $\Jac(L)=0$ if $L$ is not surjective. We have that, for any subset $V'\subset V$, $\Jac(L|_{V'})\leq \Jac(L)$.\newline
For an affine map $A:V\rightarrow W$, we set its Jacobian $\Jac(A)$ equal to the one of its linear part.

For any bounded set $K\subset V$, it holds that
\begin{equation}\label{Jacobian Area Formula}
  \Leb\bigl(A^{-1}(Y)\cap K\bigr)
  \le C_0 \frac{\Leb(Y)}{\Jac(A)}
  \qquad\text{for any Borel subset }Y\subset  W
\end{equation}
where $C_0$ is a constant that depends only the set $K$ and the dimensions of $V$ and $W$.

\begin{definition}
    For $0<\gamma\le1$, $\delta>0$ and $n\ge1$, we say that the family $T_\bt^n$ is
    $(\gamma,\delta)$-generic on $U\subset \Sone$ if the following holds. For any finite sequence
    $\{\ba_i\}_{i=1}^h$ in $\A^\infty$ such that $[\ba_i]_n$ are mutually distinct, for any
    $x\in U$ and for any integer $0<k<\gamma h$, we can choose a subsequence
    $\sigma=(\ba_{\xi(1)},\ldots,\ba_{\xi(k)})$ of length $k$, so that
    $G_{x,\sigma}$ is surjective and satisfies $\Jac(G_{x,\sigma})>\delta$.
\end{definition}

\begin{proposition}\label{generic families}
For any matrix $A\in GL(d,\R)$ with $\rho(A)<1$ and integers $\ell \geq 2$ and $n\geq 1$, there exist a finite family of smooth maps
$\left\{\psi_i\right\}_{i\in\mathcal{I}} \subset C^\infty(\Sone,\R^d)$,
such that for every
\(\phi\in C^2(\Sone,\R^d)\), the family $T_{\bt}^n$ 
is $\left(\frac1{n+1},\frac{1}{2}\right)-$generic on $\Sone$.
\end{proposition}

\begin{proof}
The proof uses the same methods of \cite{tsujii-01}, adapting them to the higher dimensional and non-conformal setting.\newline
We first fix $x\in \Sone$ and prove the existence of a finite family $\left\{\psi_i\right\}_{i\in\mathcal{I}}$ for which $T_{\bt}^n$ 
is $\left(1/({n+1}),1/2\right)-$generic on suitable neighborhood $U_x$ of $x$, rather than on the whole $\Sone$. The proposition then follows by the compactness of $\Sone$ (see how \cite[Proposition 16]{tsujii-01} implies \cite[Proposition 15]{tsujii-01}). Moreover, we notice that the linear part of $G_{x,\sigma}$ is not affected by $\phi$, hence in the following we can assume, without loss of generality, that $\phi=0$.\newline 
We define the partial order $\prec$ on $\A^n$, setting $\ba\prec\bb$ if there exists $q\geq 0$ such that $\tau^q(\bb(x))=\ba(x)$. In \cite[proof of Lemma 16]{tsujii-01}, the following fact is proved.
\paragraph{Fact:} Given any $Q\subset \A^n$, one can extract an independent set $Q'
\subset Q, $ with respect to $\prec$, with
\begin{equation*}
    |Q'| \geq \frac1{n+1} |Q|.
\end{equation*}

For $\varepsilon>0$, let $U(\varepsilon)$ be the $\varepsilon$-neighborhood of
$x$ and, for $\ba\in \A^n$, let $U_\ba(\varepsilon)$ be the $\ell^{-n}\varepsilon$-neighborhood of $\ba(x)$. Let
$\nu>0$ be an integer constant that will be specified later. For this $\nu$, we fix
$\varepsilon>0$ so small that
\begin{equation*}
    \ba\nprec \bb \quad \implies \quad
\tau^i(U_\bb(\varepsilon))\cap U_\ba(\varepsilon)= \emptyset,
\quad\text{ for }1\leq i\leq n+\nu.
\end{equation*}
Now, for $\ba\in\A^n$ and $1\leq j \leq d$, we consider functions $\psi_{(\ba,j)}\in C^{\infty}(\Sone, \R^d)$, that satisfy:
\begin{itemize}
    \item $\supp(\psi_{(\ba,j)})\subset U_\ba(\varepsilon)$;
    \item $\frac{d}{dx}\psi_{(\ba,j)}(x) = \ell^n A^{-(n-1)} \be_j$, for $x\in U_\ba(\varepsilon/3)$;
    \item $\left|\frac{d}{dx}\psi_{(\ba,j)}(x) \right| \leq 2\ell^n \|A^{-(n-1)}\|$, for $x\in U_\ba(\varepsilon)$.
\end{itemize}

\paragraph{Claim:} For $\bt\in\R^{\A^n\times\{1,\ldots,d\}}$, the family $T_\bt^n$, corresponding to the family of functions
\begin{equation*}
    \phi_\bt = \sum_{\ba\in\A^n}\sum_{j=1}^d t_{(\ba,j)} \psi_{(\ba,j)}, \quad \text{ is $ \big(1/(n+1), 1/2\big)$-generic on }U(\varepsilon/3).
\end{equation*}

\begin{proof}[Proof of the Claim]
Consider a finite sequence $\{\ba_i\}_{i=0}^h$ in $\A^\infty$ with $[\ba_i]_n$ that are mutually distinct.\newline  
By the previous discussion we can extract a subsequence $\sigma=(\ba_{\xi(1)},\ldots,\ba_{\xi(k)})$, of length $k \geq h/(n+1)$, with $[\ba_{\xi(i)}]_n\nprec [\ba_{\xi(j)}]_n$ for $1\leq i\neq j\leq k$.\newline
We define
\begin{equation*}
    E = \text{span}\left\{\be_{(\ba_{\xi{(r)}},u)} \in \R^{\A^n\times\{1,\ldots,d\}} : 1\leq r\leq k, 1\leq u \leq d\right \}.
\end{equation*}
We can identify $\R^{dk}$ with $E$ by the correspondance $\be_{rd+u}\mapsto \be_{(\ba_{\xi{(r)}},u)}$, and therefore see $G_{y,\sigma}|_E$ as a map from $E$ to itself. Finally, we denote by $(G_{y,\sigma}|_E)_{v,s}^{u,r}$ its matrix elements with respect to this basis.\newline
To recover such coefficients, recalling that $\frac d{dy}S(y,\ba_{\xi(r)};\bt) = \sum_{i=1}^\infty \ell^{i-1} A^{i-1} \frac d{dx}\phi_\bt([\ba_{\xi(r)}]_i(y))$, we compute the $u$-th coordinate of the general term of this series
\begin{equation*}
    \Big(\ell^{i-1}A^{i-1}\phi_\bt([\ba_{\xi(r)}]_i(y)) \Big)_u = \sum_{s=1}^k\sum_{v=1}^{d} t_{\ba_{\xi(s)},v} \left(\ell^{i-1} A^{i-1} \frac d{dy}\psi_{\ba_{\xi(s)},v}\big([\ba_{\xi(r)}]_i(y)\big) \right)_u,
\end{equation*}
obtaining that
\begin{equation*}
    (G_{y,\sigma}|_E)_{v,s}^{u,r} = \sum_{i=1}^{\infty}\left( \ell^{i-1}A^{i-1} \frac d{dx}\psi_{\ba_{\xi(s)},v}\big([\ba_{\xi(r)}]_i(x)\big) \right)_u.
\end{equation*}
For later convenience, we set
\begin{equation*}
    (G_{y,\sigma}^{(\nu)}|_E)_{v,s}^{u,r} = \sum_{i=1}^{n+ \nu}\left( \ell^{i-1}A^{i-1} \frac d{dx}\psi_{\ba_{\xi(s)},v}\big([\ba_{\xi(r)}]_i(x)\big) \right)_u 
    \qquad \text{and} \qquad
    G_{y,\sigma}^{(\infty)}|_E = G_{y,\sigma}|_E - G_{y,\sigma}^{(\nu)}|_E.
\end{equation*}

By the choice of the $\psi_{(\ba,j)}$, $[\ba_{\xi(r)}]_i(x)$ belongs to $\supp(\psi_{(\ba_{\xi(r)},u)})$ for $i=n$, and by the choice of $\sigma$ it cannot belong to $\supp(\psi_{(\ba_{\xi(s)},v)})$ for any $(s,v)$ for any $i\leq n+\nu$. This implies that
\begin{equation*}
    G_{y,\sigma}^{(\nu)}|_E = \text{Id}_E,
\end{equation*}
For every $y\in U(\varepsilon/3)$. At the same time, we have the estimate
\begin{align*}
    (G_{y,\sigma}^{(\infty)}|_E)_{v,s}^{u,r} &\leq \sum_{i=n+ \nu+1}^{\infty}\|A^{i-1}\|\cdot 2 \ell^{n-(n+\nu)} \|A^{-n}\|\\
    &\leq \left(C \|A^{-n}\|\|A^n\|\right) \ell^{-\nu}\sum_{i=n+ \nu+1}^{\infty}\|A^{i-1}\|, 
\end{align*}
which implies that $G_{y,\sigma}^{(\infty)}|_E$ converges to zero when $\nu\rightarrow\infty$. Therefore, in the same limit, $G_{y,\sigma}|_E$ converges to $\text{Id}_E$, implying that for $\nu$ large enough, $\Jac(G_{y,\sigma}|_E)>1/2$.
\end{proof}
This concludes the proof of the proposition.
\end{proof}

\subsection{Proof of Theorem \ref{genericity}}

We follow closely the approach of \cite{avila-gouezel-tsujii-06}.\newline
Consider a fixed $A\in GL_d(\R)$, $\ell\geq2$ and $\beta>0$, as in the statement of Theorem \ref{genericity}.\newline
Fix an integer $N_0$, whose size will be specified at the end of the proof and integers $d_0$ and $n_0$, satisfying
\begin{equation*}
    d_0/(n_0+1)>N_0+1,
  \qquad (d_0+1)\exp(-\beta n_0/2)<1/2.
\end{equation*}
We apply Proposition \ref{generic families} to $A$, $\ell$ and $n=n_0$, obtaining the functions $\psi_1,\dots,\psi_m\in C^\infty(\Sone,\R^d)$. For $\phi\in C^k(\Sone,\R^d)$, we consider the corresponding family of functions $\phi_\bt$ and maps $T_\bt^{n_0}$, for $\bt$ ranging in $[-1,1]^m$. We have that $T_\bt^{n_0}$ is $(1/(n_0+1),1/2)$-generic.\newline
We set $D_0=\sum_{i=1}^m\|\psi_i\|_{C^k}$, so that $\phi\in \mathcal{U}_{\kappa-D_0}$ implies that $\phi_\bt\in \mathcal{U}_\kappa$ for any $\bt\in[-1,1]^m$ and $\kappa> D_0$. We also pick
$p(q)$ such that $\ell^{-p(q)}\le\ell^{-q}\left\|A^{T,-q}\right\|^{-1}$. For a finite word $\bc$, let $x_\bc$ be the left endpoint of $\mathcal{P}^*(\bc)$. Finally, we define the thicker slab $\widetilde{E}_q(\be,r)=\left\{y \in \R^d : |\be \cdot y-r| < 4 \alpha \ell^{-q}\left|(A^q)^T\be\right|\right\}$.

\begin{lemma}
\label{lemma:genericity1}
    If $\phi_\bt\in\mathcal{E}(\beta,\kappa, A)$, we can take an arbitrarily large integer $q$, for which there
    exist $d_0$ words $\ba_i$, $1\le i\le d_0$, in $\A^q$, a word
    $\bc\in\A^{p(q)}$ and a slab $\widetilde{E}_q(\be,r)$ satisfying
    \begin{enumerate}
        \item\label{item:E1} $\frac{d}{dx}S_\bc(x_\bc,\ba_i;\bt)\in \widetilde{E}_q(\be,r)$, for any $1\leq i\leq d_0$;
\item\label{item:E2}$[\ba_i]_{n_0}\ne[\ba_j]_{n_0}$, if $i\neq j$.
    \end{enumerate}
\end{lemma}

\begin{proof}
By assumption, we can take an arbitrarily large $\tilde q$ such that there exist $\bc_1\in\A^{p(\tilde q)}$, a slab $E_{\tilde q}(\be_1,r_1)$ and a set of words $E\subset\A^{\tilde q}$ such that $\#E\ge\exp(\beta\tilde q)$ and
\begin{equation*}
  \exists x_\ba\in P^*(\bc)\ \text{s.t.}\
 \frac{d}{dx}S_{\bc_1}(x_\ba,\ba,\bt) \in E_{\tilde{q}}(\be_1,r_1)
  \qquad\text{for any }\ba\text{ in }E.
\end{equation*}
For each $0\le j\le[\tilde q/n_0]$, we introduce an equivalence relation $\sim_j$ on $E$
such that $\ba\sim_j \bb$ if and only if $[\ba]_{jn_0}=[\bb]_{jn_0}$, and let
\[
  \nu(j)=\max_{\ba\in E}\#\{\bb\in E\mid \bb\sim_j \ba\}.
\]
Since $\nu(0)=\#E\ge\exp(\beta\tilde q)$ while $\nu(j)\le\ell^{\tilde q-jn_0}$ obviously,
there exists $0\le j\le[\tilde q/n_0]$ such that
$\nu(j+1)<\exp(-\beta n_0/2)\nu(j)$. Let $j_*$ be the minimum of such integers $j$ and
put $q=\tilde q-n_0j_*$. Then we have $\nu(j_*)\ge\exp(\beta q)$ and
$q\ge\beta\tilde q/(2\log\ell)$. The equivalence class $H$ w.r.t. $\sim_{j_*}$ of maximum
cardinality contains at least $(d_0+1)$ non-empty equivalence classes w.r.t.
$\sim_{j_*+1}$, because
\[
  \nu(j_*)-(d_0+1)\nu(j_*+1)
  >\nu(j_*)-(d_0+1)\exp(-\beta n_0/2)\nu(j_*)>0.
\]
Thus, we can take $\bb\in\A^{\tilde q-q}$ and $\ba_i\in\A^q$, $1\le i\le d_0$,
such that $\bb\ba_i\in H$ for $1\le i\le d_0$ and that \emph{\ref{item:E2}.} holds. Put $x_{\bb\ba_i}=\bb(x_{\ba_i})$, and notice that $x_{\bb\ba_i}\in \mathcal{P}^*(\bb\bc_1)$, for $1
\leq i \leq d_0$. It follows that
\begin{equation*}
  \frac{d}{dx}S_{\bb\bc_1}(x_{\bb\ba_i},\ba_i,\bt) \in E_q(\be,r)
  \qquad\text{for }1\le i\le d_0,
\end{equation*}
where $\be= A^{T,\tilde q - q}\be_1/\|A^{T,\tilde q - q}\be_1\|$ and $r= \ell^{\tilde q - q} r_1/\|A^{T,\tilde q - q}\be_1\|$. Take $\bc\in\A^{p(q)}$ such that $\Pcal^*(\bb\bc_1)\subset\Pcal^*(\bc)$. Since the distance between $x_\bc$ and
$x_{\bb\ba_i}$ is bounded by 
$\ell^{-p(q)}\le\ell^{-q}\left\|A^{T,-q}\right\|^{-1}$,
the condition \emph{\ref{item:E1}.} follows from the bounds on the second derivatives of $S_\bc(x,\ba,\bt)$ valid when $\phi_\bt\in\mathcal{U}_\kappa$.    
\end{proof}

We define the further thickened slabs $\bar E_q(\be, r) = \left\{y \in \R^d : |\be \cdot y-r| < 8 \alpha \ell^{-q}\left\|(A^q)^T\right\|\right\}$.

\begin{lemma}
\label{lemma:genericity2}
    There exists a constant $C_{\kappa,d}>0$ such that, for every $q>0$, there exists a finite collection $\mathcal{F}_q$ of slabs of the form $\bar E_q(\be, r)$, satisfying the following property. One has
    \begin{equation*}
        \#\mathcal{F}_q\leq C_{\kappa,d} \ell^{dq} \|(A^q)^T\|^{-d}
    \end{equation*}
    and every slab $\widetilde{E}_q(\be,r)$, that intersects $B(0,\alpha)$, is contained in one of the elements of $\mathcal{F}_q$.
\end{lemma}

\begin{proof}
    Set $d\big((\be_1,r_1), (\be_2,r_2)\big)= |\be_1-\be_2| + |r_1-r_2|$. Setting $B_{\kappa,d}=\min\{1/\sqrt d, 2\alpha\}$ it is easy to verify that if $d\big((\be_1,r_1), (\be_2,r_2)\big)\leq B_{\kappa,d} \ell^{-q}\|(A^q)^T\|$ then $\widetilde E_q(\be_1, r_1)\subset \bar E_q(\be_2, r_2)$. Furthermore, the set $K\subset S^{d-1}\times\R$ of $(\be,r)$ for which $\widetilde E_q(\be, r)\cap B(0,\alpha)\neq\emptyset$ for any $q>0$ is compact. We now pick a minimal $(B_{\kappa,d} \ell^{-q}\|(A^q)^T\|)$-net in $K$ as the coordinates of the slabs of $\mathcal{F}_q$.
\end{proof}

Combining Lemma \ref{lemma:genericity1}, \ref{lemma:genericity2} and Proposition \ref{generic families}, we obtain that, for any $\bt\in[-1,1]^m$ for which $\phi_\bt\in \mathcal{E}(\beta,\kappa, A)$, there exist an arbitarily large $q>0$, a sequence of words $\sigma=(\ba_1,\ldots,\ba_{N_0})$ in $\A^q$, a word $\bc\in\A^{p(q)}$ and a slab $E \in \mathcal{F}$, such that 
\begin{equation*}
    G_{x_\bc,\sigma}(\bt)\in E^{N_0} \qquad \text{ and } \qquad \Jac(G_{x_\bc,\sigma})>\frac12.
\end{equation*}

Note that by the derivative bounds for $\phi\in\mathcal{U}_{\kappa}$ we have that $G_{x_\bc,\sigma}(\bt)\in B(0,\alpha)^N$. We set $Y(\sigma,\bc, E)= G_{x_\bc,\sigma}^{-1}\Big(\big(E\cap B(0,\alpha)\big)^{N_0}\Big)$ and
\begin{equation*}
    Y(q) = \bigcup_{\sigma\in \A^{qN_0}} \bigcup_{\bc\in\A^{p(q)}}\bigcup_{E\in\mathcal{F}} Y(\sigma,\bc, E).
\end{equation*}
Performing a union bound, and recalling \eqref{Jacobian Area Formula}, we obtain that
\begin{equation*}
    \Leb(Y(q)) \leq C \cdot \ell^{qN_0}\cdot\ell^{p(q)}\cdot \ell^{dq}\|(A^q)^T\|^{-d} \cdot \ell^{-qN_0}\left\|(A^q)^T\right\|^{N_0}.
\end{equation*}
which, for $N_0$ sufficiently large (independently from $q$), converges exponentially fast to $0$ in $q$, providing the desired conclusion by Borel Cantelli lemma.

\printbibliography
\addcontentsline{toc}{section}{Bibliography}

\end{document}